\documentclass[psamsfonts]{amsart}
\usepackage{pgf,tikz}
\usepackage{mathrsfs}
\usetikzlibrary{arrows}
\usepackage{amsmath}
\usepackage{amsthm}
\usepackage{amssymb}
\usepackage{color}
\usepackage{amscd}
\usepackage{amssymb}
\usepackage{amsfonts}
\usepackage{amsbsy}
\usepackage{graphicx}
\usepackage[dvips]{psfrag}
\usepackage{array}
\usepackage{color}
\usepackage{epsfig}
\usepackage{url}
\usepackage{overpic}
\usepackage[pagewise]{lineno}

\usepackage[hidelinks]{hyperref}
\hypersetup{colorlinks=true,linkcolor=blue,citecolor=blue,urlcolor=blue}
\usepackage{subfigure}
\usepackage[update,prepend]{epstopdf}

\makeatletter
\@namedef{subjclassname@2020}{%
	\textup{2020} Mathematics Subject Classification}
\makeatother

\usepackage[top=20mm,bottom=20mm,left=20mm,right=20mm]{geometry}

\newcommand{\R}{\ensuremath{\mathbb{R}}}

\newcommand{\dis}{\displaystyle}

\newcommand{\vs}{\vspace{0,3cm}}

\newcommand{\V}{\mathcal{V}}

\newtheorem {theorem} {Theorem}
\newtheorem {definition} {Definition}

\newtheorem {lemma}  {Lemma}
\newtheorem {example}{Example}
\newtheorem {remark} {Remark}

\begin{document}

\renewcommand{\arraystretch}{1.5}

\title[Generic one-parameter families of 3-dimensional Filippov Systems]{Generic one-parameter families of 3-dimensional Filippov Systems}
	
\author[ R. D. Euz\'{e}bio] {Rodrigo D. Euz\'{e}bio$^1$}
\address{$^1$ Instituto de Matem\'{a}tica e Estat\'{i}stica, Universidade Federal de Goi\'{a}s, Avenida Esperan\c{c}a s/n, Campus Samambaia, CEP 74690-900, Goi\^{a}nia, Goi\'{a}s, Brazil.}
\email{euzebio@ufg.br}
\email{djtonon@ufg.br}
	
\author[M. A. Teixeira] {Marco A. Teixeira$^2$}
\address{$^2$ Departamento de Matem\'{a}tica, IMECC, Universidade Estadual de Cam\-pi\-nas Rua S\'{e}rgio Buarque de Holanda 651, Cidade Universit\'{a}ria - Bar\~{a}o Ge\-ral\-do, 6065, Campinas, SP, Brazil.}
\email{teixeira@ime.unicamp.br}
	
\author[D. J. Tonon] {Durval J. Tonon$^1$}
	
\subjclass[2020]{37G05, 37G10, 37G15, 37G35, 37G40} 
\keywords{singularity, piecewise smooth vector fields, bifurcations, codimension, unfolding, structural stability}
\date{}

\maketitle
	
\begin{abstract}
This paper addresses openness, density and structural stability conditions of one-parameter families of 3D piecewise smooth vector fields (PSVFs) defined around typical singularities. Our treatment is local and the switching set, $M$, is a $2D$ surface embedded in $\mathbb{R}^3$. In short, we analyze the robustness and normal forms of certain codimension one singularities that occur in PSVFs. The main machinery used in this paper involves the theory of contact between a vector field and $M$, Bifurcation Theory and the Topology of Manifolds. Our main result states robust mathematical statements resembling the classical Kupka-Smale Theorem in the sense that we establish the openness and density of a large class of PSVFs presenting generic and quasi-generic singularities.  Due to the lack of uniqueness of certain solutions associated with PSVFs, we employ Filippov's theory as the basis of our approach throughout the paper.
\end{abstract}

\section{Introduction}
The central focus of local bifurcation theory for vector fields is to determine appropriate conditions under which the stability of trajectories near singularities can be guaranteed. These vector fields generate systems of differential equations that can be locally expressed as  $\dot{x} = f(x, \lambda)$, where \(\lambda\) is a bifurcation parameter. For foundational material on bifurcation theory, see, for instance, \cite{Aframovich1994, Wiggins2003}.

The study of local bifurcation phenomena in PSVFs in dimensions higher than two remains underdeveloped. In general, the analysis of piecewise smooth bifurcations becomes increasingly complex as the dimension of the phase space increases, leading to a proliferation of cases and subcases. Nonetheless, several real-world phenomena across different scientific disciplines are modeled by PSVFs and warrant systematic investigation. For instance, relevant applications can be found in control theory \cite{Utkin2016}, electrical and electronic systems \cite{Chillingworth2002, CristianoPaganoCarvalhoTonon2019}, biological modeling \cite{GouzeSari2003}, and alternative protocols for disease treatment \cite{deCarvalhoCristianoGoncalvesTonon2020}.

In this paper, we employ well-established results from both smooth and piecewise-smooth bifurcation theory to study local bifurcations in 3D PSVFs. There are several approaches to deal with PSVFs (see, for instance, \cite{diBernardo2008,Jeffrey2018}), two of the most commonly employed being those due to Filippov and Utkin (see \cite{Filippov1988, Utkin1977}). While the work of Utkin is more suitable for applications than Filippov's, particularly in control theory, the convention introduced by Filippov seems to be more appropriate for formal aspects of PSVFs, such as the study of bifurcations and normal forms. Ultimately, both formalisms are equivalent when the control function introduced by Utkin is linear, which makes Filippov suitable for applications in some scenarios such as control theory, relay systems, and linear systems, to cite a few; see, for instance, \cite{Bonet2017}. Due to the goals of the paper, we adopt the Filippov convention, since the majority of the background material used to construct the results also considers this convention. In addition, for the same reasons, we consider the original approach of Filippov by taking convex linear combinations of vector fields, rather than considering differential inclusions or hidden dynamics, for instance. We refer to the work of Jeffrey for a broad and modern perspective on PSVFs (see, for instance, \cite{ColomboJeffrey2011a, Jeffrey2014,  JeffreyPhysicaD2014, Jeffrey2018}). For background on the theoretical framework employed in this paper, we refer to \cite{GomideTeixeira2020, Sotomayor1974}.

We define a codimension-one surface \(M = h^{-1}(0)\), where \(0\) is a regular value of a smooth function \(h: \mathbb{R}^3 \rightarrow \mathbb{R}\). This surface divides \(\mathbb{R}^3\) into two half-spaces: the ``+'' half-space where \(h(x) \geq 0\), and the ``-'' half-space where \(h(x) \leq 0\). We denote by \(\Omega^r\) the space of PSVFs \(X = (X^+, X^-)\), endowed with the product topology. Here, \(X^+\) and \(X^-\) are vector fields defined on \(\mathbb{R}^3\), but are restricted to the ``+'' and ``-'' half-spaces respectively (see Section \ref{secao-campos-descontinuo} for details).

In the context of smooth vector fields, there is a well-developed theory offering detailed analysis of bifurcation scenarios, including normal forms, unfoldings, and codimension classification (see \cite{Dumortier1977, Peixoto1962, Sotomayor1974} for dimension two, and \cite{SotomayorTeixeira1988} for dimension three). However, for PSVFs, the available results are far more limited. In two dimensions, there has been some progress in classifying codimension zero, one, and two phenomena, as well as in constructing normal forms and establishing topological equivalence between a PSVFs and its corresponding normal form \cite{CarvalhoCardosoTonon2018, deCarvalhoTonon2014, GuardiaSearaTeixeira2011, Kuznetsov2003}. For example, Novaes et al. \cite{NovaesTeixeiraZeli2018} analyzed the bifurcation diagram of a closed trajectory connecting a two-fold singularity to itself, characterizing the scenario as a codimension-two phenomenon and providing a submersion map to describe it.

In three dimensions, some codimension zero subsets of structurally stable PSVFs were described in \cite{deCarvalhoTonon2015, TeixeiraGomide2018}, building on results from Vishik \cite{Visik1972}. A major challenge in extending bifurcation theory to three-dimensional PSVFs is the emergence of new types of singularities, notably the T-singularity (or two-fold singularity), which may lead to chaotic behavior, invariant cones, or infinite instability modules (see \cite{ColomboJeffrey2011a, GomideTeixeira2020}). Further global bifurcations and the existence of horseshoes have also been investigated in \cite{GomideTeixeira2022}.

We also address the problem of structural stability, a fundamental notion in dynamical systems theory due to its relevance in both theoretical and practical contexts. Observing different dynamical behaviors in nearby systems highlights the importance of understanding which models are structurally stable.

Finally, we note that key developments in PSVFs theory have been inspired by the study of vector fields near surface boundaries. In this regard, important contributions were made by Sotomayor and Teixeira \cite{SotomayorTeixeira1988} and Vishik \cite{Visik1972}. In particular, Vishik's classification of generic boundary points in three dimensions, establishing cusps and folds as generic singularities, has provided valuable tools and insights through the application of singularity theory.

We denote by $\Omega^r$ the set of all PSVFs, and by $\Omega^T \subset \Omega^r$ the subset of PSVFs for which \emph{both} vector fields $X^{\pm}$ are tangent to $M$ at some point. Let $\Xi_0$ be the set of PSVFs of codimension zero, and define $\Omega_1^r = \Omega^r \setminus (\Xi_0 \cup \Omega^T)$, which we call the bifurcation set. For further details, see Sections~\ref{secao-preliminares} and~\ref{secao-main-results}.

The main results of this paper identify the subset of PSVFs in $\Omega_1^r$ that is structurally stable (within $\Omega_1^r$) and show that it forms a codimension one submanifold of $\Omega^r$. We explicitly characterize the one parameter families of PSVFs that constitute this codimension one submanifold and provide normal forms and unfoldings for each family. Moreover, we prove that this submanifold is open, dense, and structurally stable in $\Omega_1^r$. Finally, to obtain a fully explicit characterization of this codimension one submanifold of $\Omega^r$, we construct a smooth map $\eta : \Omega^r \to \mathbb{R}$ that defines a submersion along each of the considered one-parameter families.

Accordingly, the main contributions of this work are as follows: $(1)$ the identification of a codimension-one submanifold of $\Omega^r$ that is structurally stable within $\Omega_1^r$; $(2)$ the derivation of normal forms and unfoldings for each associated one-parameter family;
$(3)$ the proof that this submanifold is open and dense in $\Omega_1^r$; and $(4)$ the explicit construction of a submersion $\eta : \Omega^r \to \mathbb{R}$ that characterizes each of the considered one-parameter families.

We emphasize that, to the best of our knowledge, previous results in the literature do not provide a systematic framework for bifurcation analysis in the context of 3D PSVFs.

The organization of the paper is as follows: In Section \ref{secao-preliminares} some preliminaries, definitions, and notations are presented and in Section \ref{secao-main-results} the main results of the paper are stated. In Section \ref{secao-codzeroUm} we present and discuss geometric ingredients of some subsets of PSVFs that are of codimension zero and one. In Sections \ref{proof-MainResults} and \ref{prova-teo2} we provide the proof of the main results of this paper. Finally in Appendix \ref{Apendice} we develop some results about the bifurcation theory. 



\section{Basic theory}\label{secao-preliminares}

\subsection{Vector fields near the boundary of a 3-manifold}\label{subsecao-campobordo}

Details on  the concepts and results treated in this subsection are concentrated in \cite{SotomayorTeixeira1988, Visik1972}. We assume that, locally, the boundary of $\R^3$ near to the origin is given by $M=h^{-1}(0)$ where $h:\R^3 \rightarrow \R$ is a smooth map possessing the origin as a regular point. In the following, we present the definition of germ of smooth vector fields.

\begin{definition}
Consider $Z$ a $C^r$ vector field and a germ of a vector field of class $C^r$ in a neighborhood of the origin in $\R^3$ is an equivalence class of $C^r$ vector fields defined in a neighborhood of the origin in the following way: two $C^r$ vector fields $Z_1$ and $Z_2$ are equivalent if $Z_1$ and $Z_2$ are defined in neighborhoods $U_1$ and $U_2$ of a compact neighborhood of the origin in $\R^3$, respectively, there exists a neighborhood $U_3\subset \R^3$ such that $0\in U_3\subset U_1\cap U_2$ and $Z_1|_{U_3} =Z_2|_{U_3}$.
\end{definition}
We use the notation $(\R^3,0)$ (in the germ sense) to stress that we consider local dynamics in the neighborhood of the origin. In this way, the smooth vector field $Z$ is a representant of this equivalence class. The set of all germs of $C^r$ vector fields defined in a neighborhood of the origin endowed with the $C^r $-topology, with $r\geq 1$, is denoted by $\mathfrak{X}^r$. For each $Z \in \mathfrak{X}^r$ we define a smooth function $Zh:\R^3\rightarrow\R$ by $Zh=\langle Z , \nabla h \rangle$ where $\langle . , . \rangle$ is the canonical inner product in $\R^3$. We denote $Z^nh(p)=Z(Z^{n-1}h)(p), n\geq 2$.

\begin{definition}
	Let $Z \in \mathfrak{X}^r$. We say that:
	\begin{itemize}
		\item [$(a)$] $0$ is $M$-regular point of $Z$ if $Zh(0)\neq 0$.
		\item [$(b)$] $0$ is a fold singularity of $Z$ if $Zh(0)=0$ and $Z^2h(0)\neq 0$.
		\item [$(c)$] $0$ is a cusp singularity of $Z$ if $Zh(0)=Z^2h(0)=0, Z^3h(0)\neq 0$ and the vectors $dh(0)$, $dZh(0)$ and $dZ^2h(0)$ are linearly independent.
\end{itemize}\end{definition}
The tangential set of $Z$,
$$
S_Z=\{p\in M; Zh(p)=0\},
$$
and any element in $S_Z$  is a \emph{tangential singularities} (or simply a singularity) of $Z$. In the following, we characterize the generic behavior of smooth vector fields in $(\R^3,0)$ relative to $M$. The approach is based in \cite{SotomayorTeixeira1988, Visik1972}. 

\begin{definition}
	We denote $\Sigma_0$ the set of elements $Z\in \mathfrak{X}^r$ satisfying one of the following conditions:
	\begin{itemize}
		\item [$(a)$] $0$ is a $M$-regular point of $Z$ $(Zh(0)\neq 0)$;
		\item [$(b)$] $0$ is a fold singularity of $Z$;
		\item [$(c)$] $0$ is a cusp singularity of $Z$.
	\end{itemize}
	\label{definicao-Sigma0}
\end{definition}
In \cite{SotomayorTeixeira1988} was proved that the subset $\Sigma_0$ is open, dense and it characterizes the structural stability in $\mathfrak{X}^r$. The next step is to study the bifurcation set $\mathfrak{X}^r_1=\mathfrak{X}^r-\Sigma_0$. As usual, the characterization of $\Sigma_1\subset \mathfrak{X}^r_1$, the subset composed of codimension one smooth vector fields, is based on certain issues involving unstable vector fields without rejecting a generic context. That means that certain conditions imposed on the definition of $\Sigma_0$ are violated but quasi-generic assumptions are considered.

In our approach we are focus in local dynamics, in this way quasi-generic means the condition when a point $p$ is a codimension one singularity. In other words, quasi-generic we mean that the PSVFs fails to satisfy {\it exactly} one generic condition, corresponding to a codimension-one degeneracy, while all other conditions remain generic.

We illustrate such a situation. For example, let $Z$ be off $\Sigma_0$. So we may consider the cases: $(i)$ $Z(0)= 0$ or $(ii)$ $Z(0) \neq 0$ and it does not satisfy one of the conditions $(a), (b)$ and $(c)$ in Definition \ref{definicao-Sigma0}. In the first case, we impose that $0$ is a hyperbolic singular point of $Z$ with eigenspaces transversal to the boundary. In case $(ii)$, $0$ cannot be a fold or a cusp but some extra quasi-generic conditions must be assumed.

\begin{definition}
	Call $\Sigma_1$ the set $\Sigma_1(a)\cup \Sigma_1(b)\subset \mathfrak{X}^r$, where:
	\begin{itemize}
		\item [1-] $\Sigma_1(a)\subset \mathfrak{X}^r$ is the set of smooth vector fields $Z$ such that $0$ is a hyperbolic singular point, the eigenvectors are transversal to $M$ at $0$, the eigenvalues of $DZ(0)$ have algebraic multiplicity 1. The real parts of the non-conjugated eigenvalues are distinct: if $\lambda_i\neq \lambda_j$, $\lambda_i\neq \overline{\lambda}_j$, then $Re(\lambda_i)\neq Re(\lambda_j)$, where $\lambda_i,\lambda_j$ are eigenvalues of $DZ(0)$ and $\overline{\lambda}_j$ is the conjugated of the number $\lambda_j$. 
		
		\item [2-] $\Sigma_1(b)\subset \mathfrak{X}^r$ is the set of smooth vector fields $Z$ such that $Z(0)\neq 0, (Zh)(0)=0=(Z^2h)(0)$ and one of the following conditions are valid:
		\begin{itemize}
			\item [$(b.1)$] $(Z^3h)(0)\neq 0, \dim\{dh(0), d(Zh)(0), d(Z^2h)(0)\}=2$ and 0 is a nondegenerate critical point in the Morse's sense of $Zh\mid_{M}$;
			
			\item [$(b.2)$] $(Z^3h)(0)=0, (Z^4h)(0)\neq 0$ and $0$ is a regular point of $Zh\mid_{M}$.
		\end{itemize}\label{b1}\end{itemize}
\end{definition}

Let $H(h)$ be the Hessian matrix of the function $h$. 

\begin{definition} If $Z\in \Sigma_1(a)$ then we distinguish the cases:
	\begin{itemize}
		\item [$(a.1)$] {\bf Node}: the eigenvalues of $DZ(0),\lambda_j, j=1,2,3$ are real distinct and have the same sign;
		
		\item [$(a.2)$] {\bf Saddle}: the eigenvalues of $DZ(0),\lambda_j, j=1,2,3$ are real distinct and one of them has an opposite sign compared to the others;
		
		\item [$(a.3)$] {\bf Focus}: the eigenvalues of $DZ(0)$ are $\lambda_{12}=a \pm i b,\lambda_3=c$ with $a,b,c$ distinct of zero and $a \neq c$.
	\end{itemize}
	If $Z\in \Sigma_1(b)$ then we distinguish the cases:
	\begin{itemize}
		\item [$(b.1.1)$] {\bf Lips} given in Definition \ref{b1} $(b.1)$ with $\det(H(Zh(0)))>0$, see Figure \ref{lips};
		
		\item [$(b.1.2)$] {\bf Beak to beak} given in Definition \ref{b1} $(b.1)$ with $det(H(Zh(0)))<0$, see Figure \ref{beac_to_beac};
		
		\item [$(b.1.3)$] {\bf Swallowtail} given in Definition \ref{b1} $(b.2)$, see Figure \ref{swallowtail}.
	\end{itemize}
	\label{definicao-tipos-singularidade}\end{definition}

\subsection{Piecewise smooth vector fields}\label{secao-campos-descontinuo} In the PSVFs context we consider $M$, namely the switching manifold (or discontinuity manifold), as common boundary of both smooth vector fields $X^+$ and $X^-$.  We define $M^+=\{p\in (\R^3,0); h(p)\geq 0\}$ and $M^-=\{p\in (\R^3, 0); h(p)\leq 0\}$ and we consider in the previous definitions $X(p)=X^+(p)$ or $X(p)=X^-(p)$, depending on $p\in M^+$ or $p \in M^-$, resp. The definition of a germ of PSVFs was stated in \cite{GomideTeixeira2020}. 

\begin{definition}
A $C^r$-germ of a PSVFs is an equivalence class $X=(X^+, X^-)$ of two $C^r$ vector fields defined in the following: $X_1=(X_1^+, X_1^-)$ are equivalent to $X_2=(X_2^+, X_2^-)$ if, and only if, $X_1^{\pm}$ is germ equivalent to $X_2^{\pm}$.
\end{definition}
We define a PSVFs given by 
\[
X(q)= \left\{
\begin{array}{l}X^+(q), q\in M^+\\
	X^-(q), q\in M^-,
\end{array}
\right.
\]
where $X^{\pm}\in\mathfrak{X}^r$. Call $\Omega^r=\mathfrak{X}^r\times \mathfrak{X}^r$ the space of all germs of PSVFs $X$ at $0$ endowed with the product topology and we denote by $X=(X^+,X^-)$. In $M$, we distinguish the following open regions:

\noindent$\bullet$ \textbf{Crossing region:} $M^c=\{p\in M;(X^+h)(p)(X^-h)(p)>0\}$. When convenient we denote $M^c_+=\{p\in M;$ $(X^+h)(p)>0,(X^-h)(p)>0\}$ and $M^c_-=\{p\in M;(X^+h)(p)<0,(X^-h)(p)<0\}$.

\noindent $\bullet$ \textbf{Unstable sliding region:} $M^e=\{p\in M;(X^+h)(p)>0, (X^-h)(p)<0\}$.

\noindent$\bullet$ \textbf{Stable sliding region:} $M^s=\{p\in M;(X^+h)(p)<0,(X^-h)(p)>0\}$.

\noindent Let $\mathcal{O}= M^c \cup M^e\cup M^s$. Observe that for
any $p\in \mathcal{O}$ we get $X^+(p)\neq 0$ and $X^-(p) \neq 0$ and if $p\in M^s$ then $\langle(X^- - X^+)(p),\nabla h(p)\rangle>0$.

The following definitions of trajectory-solutions at points in the swi\-tching manifold $M$ are given by Filippov's convention, stated in \cite{Filippov1988}. The \textit{sliding vector field} associated with $X$ is the vector field tangent to $M^{s,e}$. It is clear that if $q\in M^s$ then $q\in M^e$ for $-X$ and we can define the
{\it sliding vector field} on $M^e$ associated with $X$ by
$-(-X)^s$. In what follows we use the notation $\widetilde{X}^s$ for both cases. Explicitly, the vector field $\widetilde{X}^s$ is
given by
\begin{equation}\label{eq campo filippov}
	\widetilde{X}^s(p)=\dfrac{1}{X^-h(p)-X^+h(p)}( X^-h(p) X^+(p) - X^+h(p)  X^-(p)),
\end{equation}
where $X^{\pm}=(X^{\pm}_1,X^{\pm}_2,X^{\pm}_3)$. It is clear that the vector field $\widetilde{X}^s$ is topologically equivalent to the \emph{rescaling  vector field} given by 
\begin{equation}
	X^s(p)=X^-h(p) X^+(p) - X^+h(p)  X^-(p)
	\label{definicao-campodeslizante}
\end{equation}
for $p\in M^s$ and it is topologically equivalent to $-X^s(p)$ for $p\in M^e$. Therefore, either $X^s$ or $-X^s$ can be $C^r$-extended beyond the boundary of $M^s \cup M^e$ while preserving the orientation of $\widetilde{X}^s$. Throughout the paper, the trajectory of $X$ through a point $p$ on the boundary of $M^s$ is defined as the concatenation of the trajectory of $X^s$ with the orbit of $X^+$ or $X^-$ connecting to $p$, whenever such an orbit exists. Hence, trajectories are concatenated by arcs of $X^s$, $X^+$, and $X^-$. A similar situation occurs for points on the boundary of $M^e$.


\begin{definition}
	The point $p\in M^s\cup M^e$ is said a pseudo-equilibrium if $X^s(p)=0$.
\end{definition}
As said above, Filippov's convention describes three basic forms of dy\-na\-mics that would occur on the switching surface: crossing, stable sliding, and unstable sliding. 

\begin{definition}
The flow of the $X$, denoted by $\phi_X$, is obtained by the concatenation of flows of $X^+, X^-$ and $X^s$, denoted by $\phi_{X^+}, \phi_{X^-}$ and $\phi_{X^s}$,respectively.
\end{definition}

\begin{definition}
We say that $0\in M$ is:
	
$\bullet$ a regular-regular  point of $X$ if it is a regular point of both vector fields $X^+$ and $X^-$ and: $0\in M^c$ or $0\in [M^s\cup M^e]$ and $X^s(0)\neq 0$;

$\bullet$ a two-fold singularity of $X$ if it is a fold singularity of both vector fields $X^+$ and $X^-$;
	
$\bullet$ a fold-cusp singularity of $X$ if it is a cusp singularity of $X^+$ and a fold singularity of $X^-$.
\end{definition}
In the following, we distinguish the types of two-fold singularities:
\begin{itemize}
	\item [$(a)$] \emph{Parabolic}: either $[(X^+)^2h(0)\cdot (X^-)^2h(0)]>0$;
	\item [$(b)$] \emph{Hyperbolic}: $(X^+)^2h(0)>0$ and $(X^-)^2h(0)<0$;
	\item [$(c)$] \emph{Elliptic}: $(X^+)^2h(0)<0$ and $(X^-)^2h(0)>0$.
\end{itemize}
%

\subsection{The involution and the first return map associated to the two-fold singularities}\label{secao-primeiro-retorno}
The following construction is mainly based on issues developed in \cite{Teixeira1990}. Consider $X^+\in \mathfrak{X}^r$ and suppose that $X^+(0)\neq 0$ and $X^+h(0)=0$. Let $N_0$ be a transversal section to $X^+$ at $0$. Now by using the Implicit Function Theorem, there exist neighborhoods $U_{X^+}(0) \subset M$ and $V_{X^+}(0) \subset N_0$ and a smooth function $\eta_{X^+}: U_{X^+}(0)\subset M\rightarrow V_{X^+}(0)$, defined by for each $p \in U_{X^+}(0) \subset M, \eta_{X^+}(p)$ is the point where the trajectory of $X^+$ through $p$ meets $N_0$. Of course $\eta_{X^+}(0)=0$ and the singular set of $\eta_{X^+}$ coincides with $S_{X^+}$, the tangency points of $X^+$. When $\eta_{X^+}$ is a fold mapping (in the sense of the theory of the singularities of mappings) one determines the symmetric diffeomorphism $\gamma_{X^+}$ that satisfies $\eta_{X^+} \circ \gamma_{X^+}=\eta_{X^+}$. The mapping $\gamma_{X^+}$ is an involution and the fixed points of it are concentrated on $S_{X^+}$. Associated to $X^-\in \mathfrak{X}^r$ we get similarly objects such as $U_{X^-}(0), S_{X^-}, \eta_{X^-}$ and $\gamma_{X^-}$. Now given an elliptic two-fold singularity of $X=(X^+,X^-)$ we may consider the composition $\varphi_X=\gamma_{X^-}\circ \gamma_{X^+}$ defined in the intersections of domains $U_{X^+}(0)$ and $U_{X^-}(0)$, that works as first return map associated to $M$ at $0$. For more details, the references \cite{Teixeira1990} are recommended to the interested reader.

We denote by $L_X^{}$ the linear part of $\varphi_X$. Note that $\varphi_X^{}$ satisfies: $det(D\varphi_X^{})(0)$ $=1$. So generically the eigen\-va\-lues of $D\varphi(0)$ are $\beta$ and $\beta^{-1}$, which satisfy one of the following conditions:
\begin{itemize}
\item [$(a)$] $\beta\in \R$ and $\beta\neq 0$ (saddle type).

\item [$(b)$] $\beta=e^{i\,\theta}$ with $\theta\in (0,\pi)$ (elliptical type). In this case, $L_X^{}$ is a rotation.
\end{itemize}

\section{Main results}\label{secao-main-results}
In this section we state some results involving the subset $\Xi_1\subset \Omega_1^r$. For more details about the construction/definition of $\Xi_1$ and $\Omega_1^r$ see Section \ref{secao-codzeroUm}. Consider the PSVFs of the form $X=(X^+,X^-)$ having a singular, regular point for $X^+, X^-$, respectively, at the origin and satisfies the extra-conditions:
\begin{itemize}
	\item [{\bf (H)}] 
	\begin{itemize}
		\item [(a)] $X^-(0)$ is transversal to $M$ at $0$;
		\item[(b)]  the origin is a hyperbolic singular point of $X^+$ and $X^s$;
		\item[(c)] the eigenvalues of $DX^+(0)$ and $DX^s(0)$ are of algebraic multiplicity 1;
		\item[(d)] the eigenvectors associated to $DX^+(0)$ are transversal to $M$ at the origin.\end{itemize}
\end{itemize}

In the following, we consider another case in which both vector fields $X^{\pm}$ are transversal to $M$ at the origin, the switching surface coincides with $M^{s,e}$, and the origin is a singular point of the sliding vector field that is not hyperbolic; see, for instance, \cite{Sotomayor1974}.

\begin{remark}
 Note that if the origin is a hyperbolic singular point of $X^s$ then the PSVFs $X$ would be a structural stable at the origin, see Proposition 3.5 in \cite{GuardiaSearaTeixeira2011}. 
\end{remark}
Therefore, there are two cases to be considered: $(1)$ the origin is a saddle-node singular point for $X^s$ and $(2)$ the origin is a Hopf singular point for $X^s$. In the last case, we say that a PSVFs $X=(X^+,X^-)$ presenting a Hopf singular point for $X^s$ at the origin if it satisfies the following extra-conditions:
\begin{itemize}
\item [{\bf (HS)}] 
\begin{itemize}
\item[(a)] the origin is the unique singular point of $X^s$;

\item[(b)] $X^s$ has only hyperbolic periodic trajectories;

\item[(c)] The $\alpha$ and $\omega-$limit sets of any trajectory of $X$ are singular points and periodic trajectories;

\item[(d)] $X^s$ has no saddle connections.
\end{itemize}
\end{itemize}
In this way, we define $\Xi_1=\dis\bigcup_{i=1}^6\Xi_1(i)$ where

$\Xi_1(1) =\{X\in \Omega^r; 0 \mbox{ is a lips singularity of } X^{+} \mbox{ and a regular point of } X^-\}$, see Definition \ref{definicao-tipos-singularidade};

$\Xi_1(2) =\{X\in \Omega^r; 0 \mbox{ is a beak to beak singularity of } X^{+} \mbox{ and a regular point of } X^-\}$;

$\Xi_1(3) =\{X\in \Omega^r; 0 \mbox{ is a swallowtail singularity of } X^{+} \mbox{ and a regular point of } X^-\}$;

$\Xi_1(4) =\{X\in \Omega^r; 0 \mbox{ is a singular point of } X^+,  \mbox{ a regular point of } X^- \mbox{ and it satisfies conditions {\bf (H)}}\}$;

$\Xi_1(5) =\{X\in \Omega^r; 0 \mbox{ is a  regular point of } X^{\pm} \mbox{ and a saddle-node singular point of } X^s\}$.

$\Xi_1(6) =\{X\in \Omega^r; 0 \mbox{ is a  regular point of } X^{\pm}, \mbox{ a Hopf singular point of } X^s \mbox{ it satisfies conditions {\bf (HS)}}\}$.

\vspace{0,5cm}

We fix that the switching manifold $M=h^{-1}(0)=\{(x_1, x_2, x_3)\in (\R^3, 0); x_3=0\}$, i.e., $h(x_1, x_2, x_3)=x_3$. In the following we present the main results of this paper.

\begin{theorem}\label{teorema-lista-cod1}
	The subset $\Xi_1$ is a $C^{r}$ submanifold of codimension one on $\Omega^r$. In addition,
	\begin{itemize}
		\item [$(a)$] $X\in \Omega_1^r$ is $M$-stable relative to $\Omega_1^r$ if and only if $X\in \Xi_1$;
		
		\item [$(b)$] $\Xi_1$ is open and dense in $\Omega_1^r$;
		
		\item [$(c)$] there are neighborhoods $\mathcal{V}_{X} \subset \Omega^r$ of $X\in \Omega_1^r$, $U_0\subset (\R^3,0)$ of $0$ and a $C^r$ mapping $\eta:\mathcal{V}_{X} \rightarrow \R$ such that 
		\begin{itemize}
			\item [$1)$] If $\widetilde{X}\in\mathcal{V}_{X}$, $\widetilde{X}|_{U_0}$ is topologically equivalent to $X|_{U_0}$ if and only if $\eta(\widetilde{X})=0$;
			
			\item [$2)$] If $X_{\lambda_1}, X_{\lambda_2}\in \mathcal{V}_{X}$ and $\eta(X_{\lambda_1})\cdot \eta(X_{\lambda_2})>0$ then $X_{\lambda_1}|_{U_0}$ is topologically equivalent to $X_{\lambda_2}|_{U_0}$.
		\end{itemize}
	\end{itemize}
\end{theorem}

In what follows, we show that the set \(\Xi_1\) can be characterized as the union of six codimension-one \(C^r\) manifolds $\Xi_1=\dis\bigcup_{i=1}^6 \Xi_1(i)$. The next result provides the normal forms and unfolding for the subsets of $\Xi_1=\dis\bigcup_{i=1}^6 \Xi_1(i)$.

\begin{theorem}\label{teorema-formas-normais}
Consider $X_0\in \Xi_1(i), i=1, \dots, 6$. The following statements hold.
\begin{itemize}
\item[$(a)$] If $X_0 \in [\Xi_1(1) \cup \Xi_1(2) ]$ then its topological normal form is  $X_{\lambda}=(X^+_{\lambda},X^-_0)$ where $X_{\lambda}(x_{1},x_{2},x_{3}) = (1,0,x_{1}^2 + \lambda x_2  \pm x_{2}^2)$ and $X^-_{0}(x_{1},x_{2},x_{3}) = (0,0, \pm 1)$.
		
\item [$(b)$] If $X_0 \in \Xi_1(3)$ then its topological normal form is $X_{\lambda}= (X^+_\lambda, X_0^-)$ where $X^+_\lambda(x_1,x_2,x_3)=(1, 0, x_{1}^3 - x_{2} + \lambda x_{1}^2)$ and $X_0^-(x_1,x_2,x_3)=(0,0,\pm 1)$.
		

\item [$(c)$] If $X_0 \in \Xi_1(4)$ then its topological normal form is $X_{\lambda}= (X^+_\lambda, X_0^-)$ where either $X^+_\lambda(x_1,x_2,$ $x_3)=(\alpha_1 x_1,x_1+\alpha_2 x_ 2,\alpha_3(x_3-\lambda)+x_2)$ (real case) or $X^+_\lambda(x_1,x_2,x_3)=(\lambda_1 x_1,x_1+\alpha x_2+x_3,-x_2+x_3(\alpha-\lambda))$, $\alpha\neq0$ (complex case) and $X_0^-(x_1,x_2,x_3)=(a,b,c)$, $a\neq0$.

\item [$(d)$] If $X_0\in \Xi_1(5)$ then its topological normal form is $X_{\lambda}= (X^+_\lambda, X_0^-)$ where $X^+_\lambda(x_1,x_2,x_3)=(-x_2,x_1^2-\lambda,0)$, and $X^-_0(x_1,x_2,x_3)=(1,0,1)$, which corresponds to the topological normal form $X_\lambda^s$ on $M$ at $(0,0,0)$ given by $X_\lambda^s(x_1,x_2)=(-x_2,x_1^2-\lambda,0)$.
		
\item [$(e)$] If $X_0\in \Xi_1(6)$ then its topological normal form is $X_{\lambda}= (X^+_\lambda, X_0^-)$ where $X^+_\lambda(x_1,x_2,x_3)=(x_1+\lambda x_2-x_2(x_1^2+x_2^2),-2\lambda x_1+x_2+x_1(x_1^2+x_2^2),0)$, and $X^-_0(x_1,x_2,x_3)=(1,0,1)$, which corresponds to the topological normal form $X_\lambda^s$ on $M$ at $(0,0,0)$ given by $X_\lambda^s(x_1,x_2)=(\lambda x_1-x_2-x_1(x_1^2+x_2^2),x_1+\lambda x_2-x_2(x_1^2+x_1^2),0)$.
\end{itemize}	

\end{theorem}
The unfoldings of the singularities of codimension one are illustrated in Figures \ref{lips}, \ref{beac_to_beac}, \ref{swallowtail}, \ref{saddle}, \ref{Node} and \ref{Focus}, \ref{saddle-node} and \ref{hopf}, respectively.

\begin{figure}[!htb]
	\begin{center}
		\begin{overpic}[width=6. in]{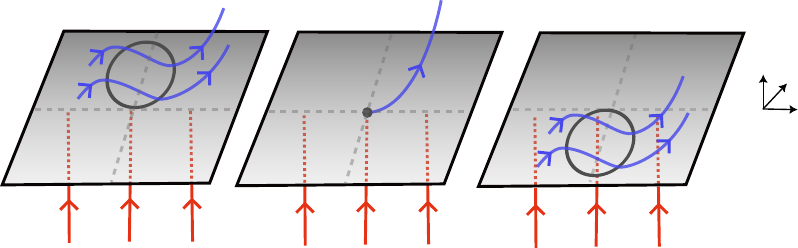}
			\put(10,28){$M$}\put(40,28){$M$} \put(70,28){$M$} \put(12,-2){$\lambda<0$}\put(41,-2){$\lambda=0$}\put(70,-2){$\lambda>0$} \put(100,16){$x_1$}\put(99,20){$x_2$}\put(95,22){$x_3$}
		\end{overpic}\hspace{0.5cm}
	\end{center}
	\caption{Unfolding of the lips singularity.}\label{lips}
\end{figure}

\begin{figure}[!htb]
	\begin{center}
		\begin{overpic}[width=6. in]{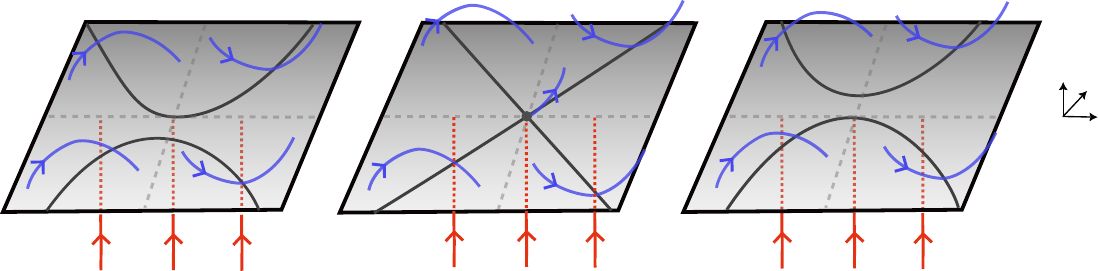}
			\put(10,24){$M$}\put(40,24){$M$} \put(71,24){$M$} \put(12,-2){$\lambda<0$}\put(43,-2){$\lambda=0$}\put(74,-2){$\lambda>0$} \put(101,13){$x_1$}\put(99,16){$x_2$}\put(96,17){$x_3$}
		\end{overpic}\hspace{0.5cm}
	\end{center}
	\caption{Unfolding of the beak to beak singularity.}\label{beac_to_beac}
\end{figure}

\begin{figure}[!htb]
	\begin{center}
		\begin{overpic}[width=6. in]{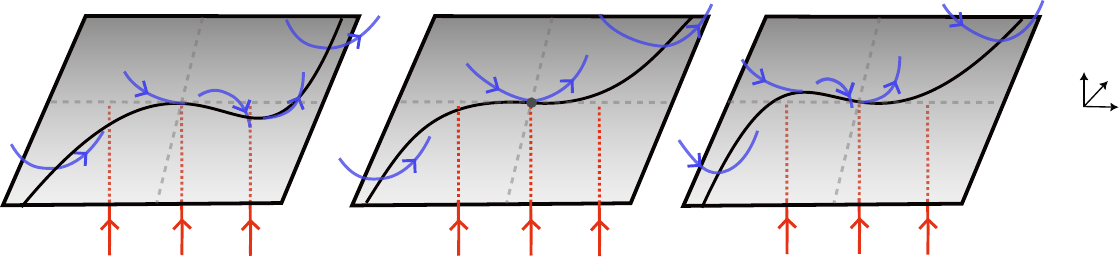}
			\put(10,22){$M$}\put(41,22){$M$} \put(71,22){$M$} \put(12,-2){$\lambda<0$}\put(42,-2){$\lambda=0$}\put(73,-2){$\lambda>0$} \put(101,13){$x_1$}\put(99,16){$x_2$}\put(96,17){$x_3$}
		\end{overpic}\hspace{0.5cm}
	\end{center}
	\caption{Unfolding of the swallowtail singularity.}\label{swallowtail}
\end{figure}

\begin{figure}[!htb]
	\begin{center}
		\begin{overpic}[width=6. in]{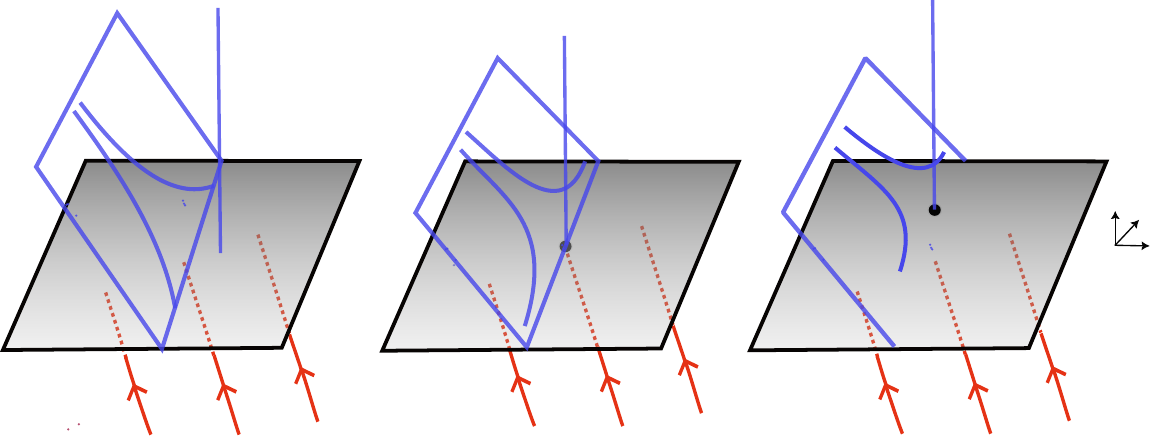}
				\put(25,25){$M$}\put(57,25){$M$}\put(89,25){$M$}\put(13,-3){$\lambda<0$}\put(46,-3){$\lambda=0$}\put(78,-3){$\lambda>0$}\put(100.5,17){$x_1$}\put(99,20){$x_2$}\put(96,21){$x_3$}
			\end{overpic}\hspace{0.5cm}
		\end{center}
		\caption{Unfolding of the saddle singularity.}\label{saddle}
	\end{figure}	
	
	\begin{figure}[!htb]
		\begin{center}
			\begin{overpic}[width=6. in]{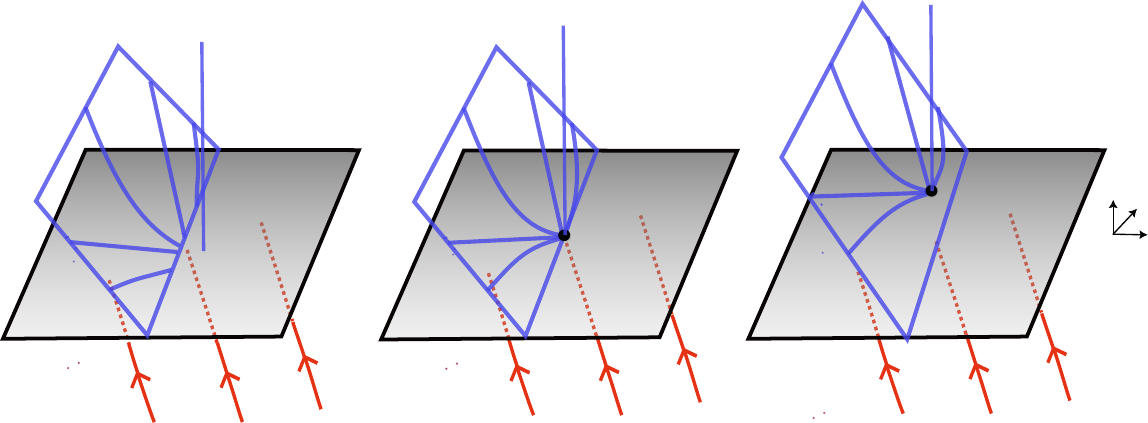}
					\put(25,25){$M$}\put(57,25){$M$}\put(89,25){$M$}\put(13,-3){$\lambda<0$}\put(46,-3){$\lambda=0$}\put(78,-3){$\lambda>0$}\put(100.5,17){$x_1$}\put(99,20){$x_2$}\put(96,21){$x_3$}
				\end{overpic}\hspace{0.5cm}
			\end{center}
			\caption{Unfolding of the node singularity.}\label{Node}
		\end{figure}	
		
		\begin{figure}[!htb]
			\begin{center}
				\begin{overpic}[width=6. in]{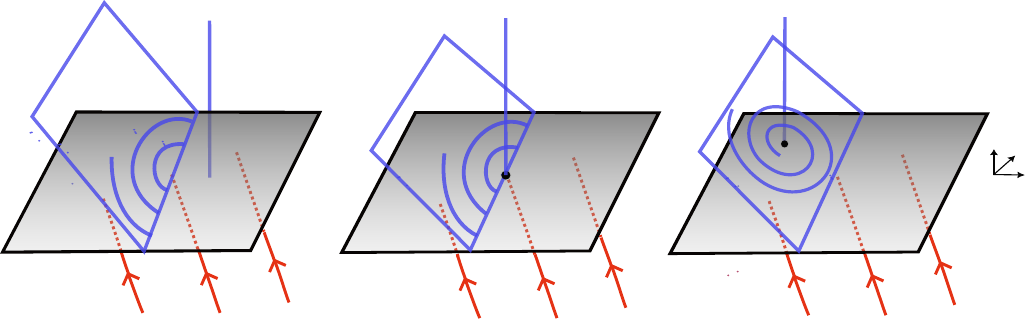}
						\put(25,22){$M$}\put(57,22){$M$}\put(89,22){$M$}\put(13,-3){$\lambda<0$}\put(46,-3){$\lambda=0$}\put(78,-3){$\lambda>0$}\put(100.5,14){$x_1$}\put(99,17){$x_2$}\put(96,18){$x_3$}
					\end{overpic}
				\end{center}
				\caption{Unfolding of the focus singularity.}\label{Focus}
			\end{figure}

			\begin{figure}[!htb]
				\begin{center}
					\begin{overpic}[width=6. in]{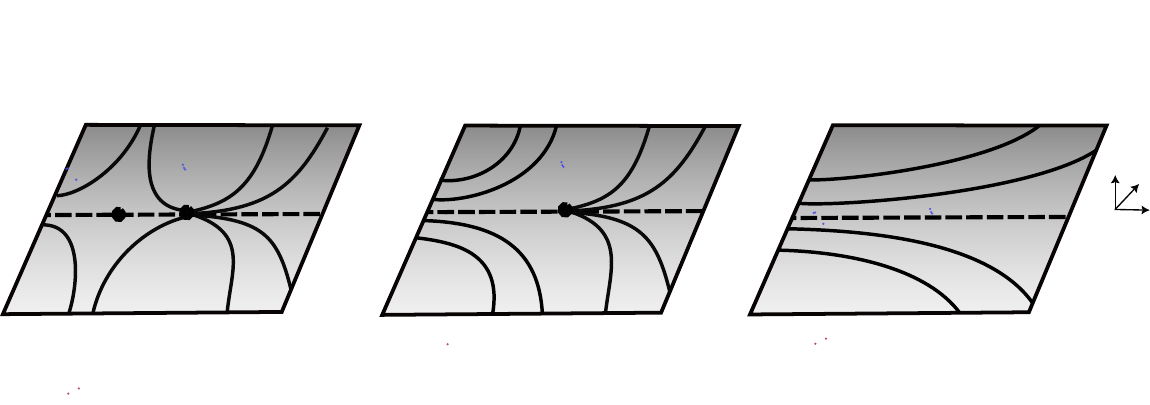}
						\put(10,26){$M$}\put(44,26){$M$}\put(76,26){$M$} \put(13,3){$\lambda<0$}\put(46,3){$\lambda=0$}\put(78,3){$\lambda>0$}\put(100.5,16.5){$x_1$}\put(99,19.5){$x_2$}\put(96,20.5){$x_3$}
					\end{overpic}
				\end{center}
				\caption{Unfolding of the saddle-node singularity.}\label{saddle-node}
			\end{figure}
			
			\begin{figure}[!htb]
				\begin{center}
					\begin{overpic}[width=6. in]{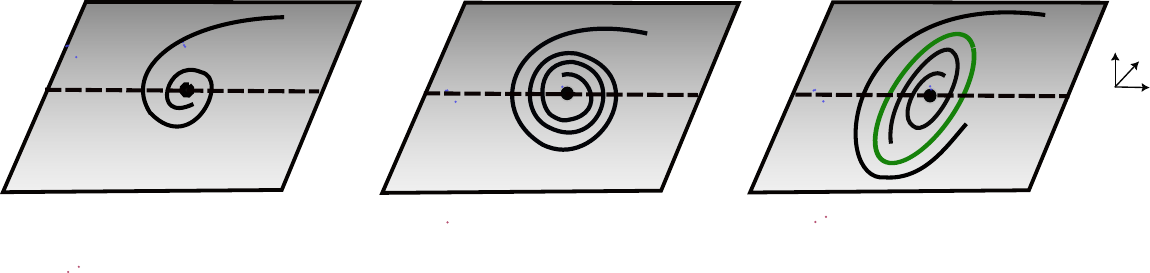}
							\put(10,24){$M$}\put(44,24){$M$}\put(76,24){$M$} \put(13,3){$\lambda<0$}\put(46,3){$\lambda=0$} \put(78,3){$\lambda>0$}\put(100.5,16.5){$x_1$}\put(99,19.5){$x_2$}\put(96,20.5){$x_3$}
						\end{overpic}
					\end{center}
					\caption{Unfolding of the Hopf singularity.}\label{hopf}
				\end{figure}

\begin{remark}
In the case where $X_0 \in \dis \bigcup_{i=1}^3 \Xi_1(i)$, if we choose a local coordinate system such that the smooth vector fields $X_0^{\pm}$ are tubular, then the corresponding unfoldings in these cases are:
	\begin{itemize}
		\item [$(a)$] If $X_0 \in [\Xi_1(1) \cup \Xi_1(2) ]$ then its topological normal form is $X_{\lambda}= (X^+_\lambda, X_0^-)$ where $X^+_\lambda(x_1,x_2,x_3)=(1,0, \lambda x_2), X_0^-(x_1,x_2,x_3)=(0,0,\pm 1)$ and $h(x_1, x_2, x_3)=x_3-(x_1^3 \pm x_1 x_2^2)$.
		
		\item [$(b)$] If $X_0 \in \Xi_1(3)$ then its topological normal form is $X_{\lambda}= (X^+_\lambda, X_0^-)$ where $X^+_\lambda(x_1,x_2,x_3)=(1,\lambda x_1, 0),$ $X_0^-(x_1,x_2,x_3)=(0,0,\pm 1)$ and $h(x_1, x_2, x_3)=x_3-(x_1^4 + x_1 x_2)$.
	\end{itemize}
\end{remark}

\section{Subsets of codimension zero and one of 3D PSVFs}\label{secao-codzeroUm}

\subsection{Classification of singularities of codimension zero of $X\in\Omega^r$}\label{subsecao-classificacao-sing}

If the origin is a two-fold singularity, then the flow of $X^+$ can induce an involution $\gamma_{X^+}\colon M \to M$, see Subsection~\ref{secao-primeiro-retorno}. In this case, there exist local connections between the regions $M^s$ and $M^e$ through the trajectories of $X^+$, and certain transversality conditions must be satisfied along these connections. If the origin is an elliptic two-fold singularity, then a first return map $\varphi_X$ is defined in a neighborhood of the origin. For the case in which the origin is a two-fold singularity, as stated in \cite{GomideTeixeira2020}, we define the following additional conditions:

\begin{itemize}
	\item [{\bf (P)}] If the origin is a parabolic two-fold singularity of $X\in \Omega^r$, then the germ of the involution $\gamma_{X^+}$ at 0 associated to $X^+$ satisfies:
	\begin{itemize}
		\item [(a)] $\gamma_{X^+}(S_{X^-}) \pitchfork S_{X^-}$ at the origin;
		\item [(b)] $X^s$ and $\gamma_{X^+}^*(X^s)$ are transversal at each point of $M^s\cap \gamma_{X^+}(M^e)$; 
		\item [(c)] $\gamma_{X^+}(S_{X^-}) \pitchfork X^s$ at the origin.
	\end{itemize}
	
	\item [{\bf (E)}] If the origin is an elliptic two-fold singularity of $X\in \Omega^r$, then the germ of the first return map $\varphi_X$ at 0 associated to $X$ has a fixed point at 0 of saddle type with both local invariant manifolds $W^{s,u}_{loc}$ contained in $M^c$.
\end{itemize}
In the following, we present, by considering the intrinsic properties of each case, the subsets of codimension zero and one PSVFs. We define the subset
\[
\Xi_0 = \dis \bigcup_{i=1}^4 \Xi_0(i),
\]
where

$\Xi_0(1) =\{X\in \Omega^r; 0 \mbox{ is a regular point of both } X^{\pm}\}$;

$\Xi_0(2) =\{X\in \Omega^r; 0 \mbox{ is a fold singularity of } X^{+} \mbox{ and a regular point of } X^-\}$;

$\Xi_0(3) =\{X\in \Omega^r; 0 \mbox{ is a cusp singularity of } X^{+} \mbox{ and a regular point of } X^-\}$;

$\Xi_0(4) =\{X\in \Omega^r; 0 \mbox{ is a two-fold singularity of } X \mbox{ and it satisfies the conditions {\bf (P)} and {\bf (E)}} \}$.


%
%
%
%

\subsection{Some subsets of codimension one of $X\in\Omega^r$}\label{subsecao-classificacao-sing-cod1}

From \cite{TeixeiraGomide2018}, the set $\Xi_0$ is locally structurally stable at the origin. Moreover, it is not a residual set in $\Omega^r$. We intend to violate each of the conditions that define the set $\Xi_0$, in a way similar to what was done in Subsection~\ref{subsecao-campobordo}.

Therefore, in the following analysis, we consider the subsets of PSVFs that violate, in a stable way, each of the previous conditions defining the subset $\Xi_0$, assuming that at least one of the vector fields is transversal to the switching surface $M$. We denote by $\Omega^T$ the set of PSVFs $X=(X^+,X^-)$ for which \emph{both} $X^{\pm}$ are tangent to $M$ at the origin. The analysis of codimension-one subsets in $\Omega^T$, for instance the case in which the origin is a fold--cusp singularity, will be carried out in a forthcoming work. In this way, we obtain a collection of sets in
\[
\Omega_1^r=\Omega^r\setminus\bigl(\Xi_0 \cup \Omega^T\bigr).
\]
One of the purposes of this work is to provide an open dense set in $\Omega_1^r$, namely $\Xi_1$.

We emphasize that, in order to define the codimension-one submanifold $\Xi_1 \subset \Omega^r$ stated in Section~\ref{secao-main-results}, we impose certain conditions on the eigenvalues and eigenvectors of $DX^+(0)$ and $DX^s(0)$. In what follows, we present results and examples that support a discussion of which generic conditions imposed on the smooth vector field $X^+$ are inherited by the sliding vector field $X^s$ (see Lemmas~\ref{lemma_equlibrio_Filippov} and~\ref{lemma_transversalidade}) and, on the other hand, which generic conditions are not inherited by $X^s$ (see Example \ref{Exemplo-Singular}).


We choose local coordinates $x=(x_1,x_2,x_3)$ around the origin such that $X^+=(X_1^+,X_2^+,X_3^+)$ satisfies $X_3(x)=x_2$ and $h(x)=x_3$, so $M=h^{-1}(0)$ is locally a hyperplane containing the origin.

\begin{lemma}\label{lemma_equlibrio_Filippov}
If $p\in M$ is a singular point of $X^+$ then it is a singular point of $X^s$.
\end{lemma}

The proof of the previous result is straightforward and it follows directly from the expression \eqref{definicao-campodeslizante}.

\begin{lemma}\label{lemma_transversalidade}
If the eigenspaces of $DX^+(p)$ are transversal to $M$ at $p$ then the eigenspaces of $DX^s(p)$ are transversal to $S_{X^+}$, the boundary of $X^s$.
\end{lemma}

\begin{proof}
We set $p=(0,0,0)$ as being the origin. From \cite{SotomayorTeixeira1988} if the eigenspaces of $DX^+(p)$ are transversal to $M$ at the origin then $\frac{\partial X_2^+}{\partial x_1}(0,0,0)\neq0$. The sliding vector fields are given by 
	$$
	X^s(x_1,x_2)=(X^+_1(x_1,x_2,0) X^-_3(x_1,x_2,0)-x_2X^-_1(x_1,x_2,0),X^+_2(x_1,x_2,0) X^-_3(x_1,x_2,0)-x_2X^-_2(x_1,x_2,0)).
	$$
	Note that $X^s(0,0)=(0,0)$ and the linearized vector field of $X^s(x_1,x_2)$ at $(0,0)$ is $A^s.(x_1,x_2)^T$ where
	$$
	A^s=D_X^s(0,0)=
	\left(
	\begin{array}{cc}
		X^-_3(0,0,0)\frac{\partial X^+_1}{\partial x_ 1}(0,0,0) & -X^-_1(0,0,0)+X^-_3(0,0,0)\frac{\partial X^+_1}{\partial x_ 2}(0,0,0)\vspace{0.2cm}\\
	X^-_3(0,0,0)\frac{\partial X^+_2}{\partial x_ 1}(0,0,0) & -X^-_2(0,0,0)+X^-_3(0,0,0)\frac{\partial X^+_2}{\partial x_ 2}(0,0,0)
	\end{array}
	\right).
	$$
The boundary $S_{X^+}$ of $X^s$ is the fold line associated with $X^+$, that is, $S_{X^+}=\{(x_1,x_2)\,;\,x_2=0\}$. Therefore, the eigenvectors $v_1$ and $v_2$ of $A^s$ are transversal to $S_{X^+}$ if they satisfy $\langle \widetilde{v}, v_i \rangle \neq 0$, $i=1,2$, where $\widetilde{v}=(1,0)$ is the gradient of the curve $(x_1,0)$ on $M$. After computing $v_1$ and $v_2$, one can see that $\langle \widetilde{v}, v_i \rangle \neq 0$ for both $i=1$ and $i=2$ provided that one of the following conditions holds:
\begin{itemize}
	\item[(i)] $X^-_3(x,y,0)=0$, which cannot occur because, in this case, the
	trajectories of $X^-$ are contained in $M$;
	\item[(ii)] $\dfrac{\partial X_2^+}{\partial x_1}(0,0,0)=0$, which cannot occur
	by hypothesis;
	\item[(iii)] $X^-_1(x,y,0)=X^-_3(x,y,0)\dfrac{\partial X_3^+}{\partial x_2}(0,0,0)=0$,
	which cannot occur because otherwise $X^s$ would have a continuum of equilibrium
	points through $(0,0,0)$.
\end{itemize}
Therefore, the result is proved.

\end{proof}

The next example shows that other generic conditions of $X^+$ are not induced to $(X^+)^s$ as the existence of singular points and transversality are. More precisely, we prove that the origin is a isolated singular point for $X^+$ but that it not inherited by the vector field $X^s$ which has a continuum of equilibrium points.

\begin{example}\label{Exemplo-Singular}
Consider the linear PSVFs $X=(X^+,X^-)$ with $X^\pm=(X_1^\pm,X_2^\pm,X_3^\pm)$ where  $X^+(x_1,x_2,x_3)=(x_2,  -2x_1-x_2, x_2+3x_3)$ and $ X^-(x_1,x_2,x_3)=(a, b,c)$. We consider $M=\{x_3=0\}$ therefore $c\neq0$ to $X^-$ be transversal to $M$ at the origin. Without loss of generality we set $c<0$. Clearly, the origin is an isolated singular point for $X^+$. Moreover, it is of focus-saddle type since the eigenvalues associated with the linearized vector field (associated to $X^+$) have eigenvalues $\lambda_{1,2}=\frac{1}{2}(-1+i\sqrt{7})$ and $\lambda_3=3$. The extended sliding vector field defined on $M$ which is $(X^+)^s(x_1,x_2)=((-a+c)x_2,-2cx_1-(b+c)x_2)$.
	
As we have seen in Lemma \ref{lemma_equlibrio_Filippov}, since the origin is a singular point for $X^+$, then it is also a singular point for $(X^+)^s$. However, we notice that it is isolated if, and only if, $a\neq c$ and $c\neq0$. If $a=c$ or $c=0$, then there exists a continuum of singular points over the straight line $x_1=-\frac{(b+c)x_2}{2c}$ or $x_2=0$, respectively.
\end{example}

The previous example shows that there exists PSVFs for which their generic conditions on hyperbolicity/transversa\-li\-ty/algebraic multiplicity one of the singular points of $X^+$ are not induced to $X^s$. In this way, the extra conditions on the vector field $X^s$ given in condition {\bf (H)}, stated in Section \ref{secao-main-results} are necessary.

In the next examples, we show that $\Xi_1(5)$ and $\Xi_1(6)$ are not empty.

\begin{example} [Non-hyperbolic saddle-node singular point]
Consider the linear PSVFs $X=(X^+,X^-)$ where
$$
\begin{array}{ll}
X^+(x_1,x_2,x_3)   &=\left(a_1 x_1+a_2 x_3+a_1(a_1+a_3),a_4 x_3,a_5 x_3-1\right),\\
X^-(x_1,x_2,x_3)   &=\left(\mu+a_3 x_1+a_6 x_3-a_1(a_1+a_3),-x_2+a_7 x_3,-\dfrac{1}{a_1}x_1+a_8x_3+1\right),
\end{array}
$$
with $a_1\neq0$ and $M=\{x_3=0\}$. Notice that $X^+(0,0,0)=(a_1(a_1+a_3),0,-1)$ and $X^-(0,0,0)=(\mu-a_1(a_1+a_3),0,1)$, that is, $X^+$ and $X^-$ are regular at $(0,0,0)$. Moreover, we get $X^+h(0,0,0)=-1$ and $X^-h(0,0,0)=1$ so $(0,0,0)$ is a sliding point since it satisfies the condition $X^+h(0,0,0)X^-h(0,0,0)=-1<0$. In particular, $X^{\pm}(0,0,0)$ are transversal ot $M$ at $(0,0,0)$. The extended sliding vector field $X^s$ writes $X^s(x_1,x_2)=(\mu-x_1^2,-x_2)$, that is, $X^s$ has an saddle-node singular point at $(0,0)$ for $\mu=0$, that is, a quasi-generic singular point. Therefore, $X\in\Xi_1(5)\neq\emptyset$.
\end{example}

\begin{example}[Hopf singular point]
Consider $M=\{x_3=0\}$ and the PSVFs $X=(X^+,X^-)$, where 
\[
\begin{array}{ll}
X^+(x_1,x_2,x_3)    &=(-y+x(\mu-1), x+ y(\mu-1), x^2+y^2-1), \\ X^-(x_1,x_2,x_3)    &=(x, y, 1).
\end{array}
\]
 The sliding vector field $X^s(x_1,x_2)=(\mu x_1-x_2-x_1(x_1^2+x^2),x_1+\mu x_2-x_2(x_1^2+x_2^2))$, which is a classical normal form for the Hopf bifurcation occurring at the origin for $\mu=0$.
\end{example}


\section{Proof of Theorem \ref{teorema-lista-cod1}}\label{proof-MainResults}

\subsection{Construction of the submersion}\label{secao-submersao}
In each of the following cases, we exhibit explicitly the submersion that characterizes the codimension one submanifold of PSVFs for each case. The construction was inspired in the work \cite{SotomayorTeixeira1988}. 

\subsubsection{Lips and beak to beak cases.} We consider $X_{0}=(X_{0}^+,X_{0}^-)\in \Xi_{1}(1)\cup \Xi_{1}(2)$. Therefore, $X_{0}^+h(0)=(X_{0}^+)^{2}h(0)=0$, $(X_{0}^+)^{3}h(0)\neq 0$, and $\{Dh(0),DX_{0}^+h(0),D(X_{0}^+)^{2}h(0)\}$ is linearly dependent. Moreover, the origin is a non-degenerate point of $X_{0}^+h$, and hence we have:

\begin{equation*}
	\begin{array}{lllll}
		X_1^+(0) \neq 0,  & X_2^+(0)=0, & X_3^+(0)=0, & \dfrac{\partial X_3^+}{\partial x_{1}}(0)=0,  &\dfrac{\partial^{2} X_3^+}{\partial x_{2}\partial x_{3}}(0) = 0,  \\ \\
		\dfrac{\partial X_3^+}{\partial x_{2}}(0)=0, 
		&\dfrac{\partial^{2} X_3^+}{\partial x_{1}\partial x_{2}}(0) = 0,  & \dfrac{\partial^{2} X_3^+}{\partial x_{1}^2 (0)}\neq 0, & \dfrac{\partial^{2} X_3^+}{\partial x_{2}^2 (0)} \neq 0. &
	\end{array}
\end{equation*}
We define 
$$
\begin{array}{ccc}
	\hspace{0,6 cm}\xi: [\Xi_1(1) \cup \Xi_1(2)] \times \mathbb{R}^2 \rightarrow \mathbb{R}\\
	\hspace{-0,3 cm} (X,(x_{1},x_{2}))  &\hspace{-1,2 cm} \mapsto &  \det [Dh(x_{1},x_{2},0),DX^+h(x_{1},x_{2},0),D(X^+)^{2}h(x_{1},x_{2},0)].
\end{array}
$$
Since $0$ is a beak to beak or lips singularity of $X_0$, we have $\xi(X_{0},0)=0$. On the other hand, 

\[
D\xi(X_{0},0)= \left( 0, -X_1^+(0) \dfrac{\partial^{2} X_3^+}{\partial x_{1}^{2}}(0) \dfrac{\partial^{2} X_3^+}{\partial x_{2}^2}(0), 0\right) \neq 0.
\]

Thus, by the Implicit Function Theorem, there exist neighborhoods $\mathcal{V}_{X_{0}}$ of $X_{0}$ in $\Omega_1^r$, $\mathcal{U}$ of $0 \in \mathbb{R}^2$ and a unique $C^r$ map $\widetilde{\alpha}:\mathcal{V}_{X_{0}} \rightarrow \mathcal{U} \subset \mathbb{R}^2$, satisfying $\widetilde{\alpha}(X)=0$, such that $\xi(X,\widetilde{\alpha}(X))=0$. Therefore $p_{X}(\widetilde{\alpha}(X),0)$ is a lips or beak to beak singularity of $X$. Consider the map
$$
\begin{array}{ll}
	\eta:  &\mathcal{V}_{X_{0}}  \rightarrow \mathbb{R}\\
	& X \mapsto  \eta(X)=\xi(X,p_{X}).
\end{array}
$$
Note that $\Xi_{1}(1) \cup \Xi_{1}(2)=[\eta^{-1}(0)\cap \mathcal{V}_{X_{0}}]$. It remains to prove that for all $X=(X^+,X^-) \in \eta^{-1}(0)\cap \mathcal{V}_{X_{0}}$ we have that $d \eta(X)$ is surjective. Consider $\gamma: (-\lambda_{0},\lambda_{0}) \rightarrow \Omega_{1}$ a smooth curve given by
\begin{equation*}
\gamma(\lambda)=X_{\lambda}=(X_{0}^+ + \lambda(0,0,x_2), X_{0}^-).
\end{equation*}\label{desd_beak_to_beak_lips_rev_1-3}
Note that $\eta(X_{0})=0$ and 
\[
\begin{array}{ll}
	\eta(\gamma(\lambda))     &=  - \left( \frac{\partial X_3^+}{\partial x_2} \right)^2 
	\frac{\partial X_2^+}{\partial x_1} 
	+ \frac{\partial X_3^+}{\partial x_1} \Bigg(
	X_3^+ \frac{\partial^2 X_3^+}{\partial x_2 \partial x_3} 
	+ X_2^+ \frac{\partial^2 X_3^+}{\partial x_2^2} 
	+ \frac{\partial X_1^+}{\partial x_2} \frac{\partial X_3^+}{\partial x_1} 
	+ X_1^+ \frac{\partial^2 X_3^+}{\partial x_1 \partial x_2}
	\Bigg) \\[6pt]
	&\quad + \frac{\partial X_3^+}{\partial x_2} \Bigg(
	\Big( \frac{\partial X_2^+}{\partial x_2} - \frac{\partial X_1^+}{\partial x_1} \Big) \frac{\partial X_3^+}{\partial x_1}
	- X_3^+ \frac{\partial^2 X_3^+}{\partial x_1 \partial x_3}
	- X_2^+ \frac{\partial^2 X_3^+}{\partial x_1 \partial x_2}
	- X_1^+ \frac{\partial^2 X_3^+}{\partial x_1^2}
	\Bigg) \\[6pt]
	&\quad + \lambda \Bigg[
	\Big( \frac{\partial X_2^+}{\partial x_2} 
	+ x_2 \frac{\partial^2 X_3^+}{\partial x_2 \partial x_3} 
	- \frac{\partial X_1^+}{\partial x_1} \Big) 
	\frac{\partial X_3^+}{\partial x_1}
	- X_3^+ \frac{\partial^2 X_3^+}{\partial x_1 \partial x_3} \\
	&\qquad\quad
	- \frac{\partial X_3^+}{\partial x_2} \left( 
	2 \frac{\partial X_2^+}{\partial x_1} 
	+ x_2 \frac{\partial^2 X_3^+}{\partial x_1 \partial x_3}
	\right)
	- X_2^+ \frac{\partial^2 X_3^+}{\partial x_1 \partial x_2}
	- X_1^+ \frac{\partial^2 X_3^+}{\partial x_1^2}
	\Bigg] 
	+ O(\lambda^2).
\end{array}
\]
Therefore,
\[
\begin{array}{ll}
	d \eta(X) \cdot \gamma'(0)   &=\displaystyle\lim_{\lambda  \to 0}\dfrac{\eta(\gamma(\lambda))-\eta(Z_{0})}{\lambda} =
	-X_1^+(0)\dfrac{\partial^{2} X_3^+}{\partial x_{1}^2}(0) \neq 0,
\end{array}
\]
concluding the proof in this case. 	

\subsubsection{Swallowtail case.} We consider $X_{0}=(X_{0}^+,X_{0}^-) \in \Xi_1(3)$. Therefore $X^+_{0}h(0)=(X^+_{0})^{2}h(0)=(X^+_{0})^{3}h(0)=0$, $(X^+_{0})^{4}h(0) \neq 0$ and the origin is a regular point of $X^+_{0}h$, therefore we have:
\begin{equation*}
	\left.
	\begin{array}{lllll}
		X_1^+(0) \neq 0, & X_1^+(0)=0, & X_3^+(0)=0, & 	\dfrac{\partial X_3^+}{\partial x_{1}}(0)=0,  &\dfrac{\partial X_1^+}{\partial x_{1}}(0)=0, \\
		\dfrac{\partial X_3^+}{\partial x_{3}}(0)=0,   &\dfrac{\partial X_3^+}{\partial x_{2}}(0)\neq 0,    &\dfrac{\partial^{2} X_1^+}{\partial x_{1}^2}(0)=0,     &\dfrac{\partial^{2} X_3^+}{\partial x_{1}^2}(0)=0,   &\dfrac{\partial^{3} X_3^+}{\partial x_{1}^{3}}(0)\neq 0.
	\end{array}	\right.
\end{equation*}
We define 
$$
\begin{array}{ccc}
	\hspace{0,6 cm}\xi: \Omega_{1} \times \mathbb{R}^2 \rightarrow \mathbb{R}\\
	\hspace{-0,3 cm} (X,(x_{1},x_{2}))  &\hspace{-1,2 cm} \mapsto &  (X^+)^{3}h(x_{1},x_{2},0).
\end{array}
$$
Since $0$ is a swallowtail singularity of $X$, we have $\xi(X_{0}^+,0)=0$. On the other hand, 
\[
\begin{array}{ll}
D\xi(X_{0},0)    &=  \Bigg( X_1^+(0)^2  \dfrac{\partial^3 X_3^+}{\partial x_{1}^3}(0), \left( \dfrac{\partial X_1^+}{\partial x_{2}}(0) \right)^2  \dfrac{\partial X_3^+}{\partial x_{2}}(0) + \dfrac{\partial X_1^+}{\partial x_{3}}(0) \left( \dfrac{\partial X_3^+}{\partial x_{2}}(0) \right)^2 + 2 X_1^+(0) \dfrac{\partial X_1^+}{\partial x_{2}}(0)\, \dfrac{\partial^2 X_3^+}{\partial x_{1}\partial x_{2}}(0) \\
	&+ X_1^+(0)\, \dfrac{\partial X_3^+}{\partial x_{2}}(0) \Biggl( 2 \dfrac{\partial^2 X_3^+}{\partial x_{1}\partial x_{3}}(0) + \dfrac{\partial^2 X_1^+}{\partial x_{1}\partial x_{2}}(0) \Biggr) + \dfrac{\partial X_1^+}{\partial x_{1}}(0) \, \dfrac{\partial^2 X_3^+}{\partial x_{1}\partial x_{2}}(0) + X_1^+(0) \, \dfrac{\partial^3 X_3^+}{\partial x_{1}^2 \partial x_{2}}(0) , \\
	& \dfrac{\partial X_1^+}{\partial x_{3}}(0) 	\Biggl( \dfrac{\partial X_1^+}{\partial x_{2}}(0) \, \dfrac{\partial X_3^+}{\partial x_{2}}(0) + 2 X_1^+(0)   \dfrac{\partial^2 X_3^+}{\partial x_{1}\partial x_{2}}(0) \Biggr) + X_1^+(0)\, \dfrac{\partial X_3^+}{\partial x_{2}}(0)\, \dfrac{\partial^2 X_1^+}{\partial x_{1}\partial x_{3}}(0)  
	+ \dfrac{\partial X_1^+}{\partial x_{1}}(0) \, \dfrac{\partial^2 X_3^+}{\partial x_{1}\partial x_{3}}(0) \\
	&+ X_1^+(0)\, \dfrac{\partial^3 X_3^+}{\partial x_{1}^2 \partial x_{3}}(0)
	\Bigg) \neq 0.
\end{array}
\]
Thus, by the Implicit Function Theorem, there exist neighborhoods $\mathcal{V}_{X_{0}}$ of $X_{0}$ in $\Omega_{1}^r$, $\mathcal{U}$ of $0 \in \mathbb{R}^2$ and a unique $C^r$ map $\widetilde{\alpha}:\mathcal{V}_{X_{0}} \rightarrow \mathcal{U} \subset \mathbb{R}^2$, satisfying $\widetilde{\alpha}(X)=0$, such that $\xi(X,\widetilde{\alpha}(X))=0$. Therefore $p_{X}(\widetilde{\alpha}(X),0)$ is a swallowtail singularity of $X$. Consider the map
$$
\begin{array}{ccc}
	\hspace{0,6 cm}\eta: \mathcal{V}_{X_{0}} \hspace{0,5 cm} \rightarrow \mathbb{R}\\
	\hspace{-0,3 cm} X  &\hspace{-1,2 cm} \mapsto &  \eta(X)=\xi(X,p_{X}).
\end{array}
$$
Note that $\Xi_{1}(3)=[\eta^{-1}(0)\cap \mathcal{V}_{X_{0}}]$. It remains to prove that for all $X=(X^+,X^-) \in \eta^{-1}(0)\cap \mathcal{V}_{X_{0}}$ we have that $d \eta(X)$ is surjective. Define $\gamma: (-\lambda_{0},\lambda_{0}) \rightarrow \Omega^r$ a smooth curve given by
\begin{equation}\label{desd_rabo-andor_rev_1-3}
\gamma(\lambda)=X_{\lambda}=(X_{0}^+ +\lambda (0, 0, x_1^2), X_{0}^-).
\end{equation}
Consider \eqref{desd_rabo-andor_rev_1-3}, note that $\eta(X_{0})=0$ and
\[
\begin{array}{ll}
	\eta(\gamma(\lambda))    &= X_3^+ \Bigg[ 
	\left(\frac{\partial X_3^+}{\partial x_3}\right)^2
	+ X_3^+ \frac{\partial^2 X_3^+}{\partial x_3^2}
	+ \frac{\partial X_2^+}{\partial x_3} \frac{\partial X_3^+}{\partial x_2} 
	+ X_2^+ \frac{\partial^2 X_3^+}{\partial x_2 \partial x_3} 
	+ \frac{\partial X_1^+}{\partial x_3} \frac{\partial X_3^+}{\partial x_1} 
	+ X_1^+ \frac{\partial^2 X_3^+}{\partial x_1 \partial x_3} 
	\Bigg] \\
	&+ X_2^+ \Bigg[
	\frac{\partial X_3^+}{\partial x_3} \frac{\partial X_3^+}{\partial x_2} 
	+ \frac{\partial X_2^+}{\partial x_2} \frac{\partial X_3^+}{\partial x_2} 
	+ X_3^+ \frac{\partial^2 X_3^+}{\partial x_2 \partial x_3} 
	+ X_2^+ \frac{\partial^2 X_3^+}{\partial x_2^2} 
	+ \frac{\partial X_1^+}{\partial x_2} \frac{\partial X_3^+}{\partial x_1} 
	+ X_1^+ \frac{\partial^2 X_3^+}{\partial x_1 \partial x_2} 
	\Bigg] \\
	&+ X_1^+ \Bigg[
	\frac{\partial X_3^+}{\partial x_2} \frac{\partial X_2^+}{\partial x_1} 
	+ \frac{\partial X_3^+}{\partial x_3} \frac{\partial X_3^+}{\partial x_1} 
	+ \frac{\partial X_1^+}{\partial x_1} \frac{\partial X_3^+}{\partial x_1} 
	+ X_3^+ \frac{\partial^2 X_3^+}{\partial x_1 \partial x_3} 
	+ X_2^+ \frac{\partial^2 X_3^+}{\partial x_1 \partial x_2} 
	+ X_1^+ \frac{\partial^2 X_3^+}{\partial x_1^2} 
	\Bigg] \\
	&+ \Bigg[ 2 (X_1^+)^2 
	+ 2 x_1 X_3^+ \frac{\partial X_1^+}{\partial x_3} 
	+ 2 x_1 X_1^+ \frac{\partial X_3^+}{\partial x_3} 
	+ x_1^2 \left(\frac{\partial X_3^+}{\partial x_3}\right)^2
	+ 2 x_1^2 X_3^+ \frac{\partial^2 X_3^+}{\partial x_3^2} 
	+ 2 x_1 X_2^+ \frac{\partial X_1^+}{\partial x_2} \\
	&+ x_1^2 \frac{\partial X_2^+}{\partial x_3} \frac{\partial X_3^+}{\partial x_2} 
	+ 2 x_1^2 X_2^+ \frac{\partial^2 X_3^+}{\partial x_2 \partial x_3} 
	+ 2 x_1 X_1^+ \frac{\partial X_1^+}{\partial x_1} 
	+ x_1^2 \frac{\partial X_1^+}{\partial x_3} \frac{\partial X_3^+}{\partial x_1} 
	+ 2 x_1^2 X_1^+ \frac{\partial^2 X_3^+}{\partial x_1 \partial x_3}	\Bigg] \lambda + \mathcal{O}(\lambda^2)
\end{array}
\]
and derivating the map $\eta$ along the curve $\gamma(\lambda)$ we get
$$
d \eta(X) \cdot \gamma'(0)=\displaystyle\lim_{\lambda  \to 0}\dfrac{\eta(\gamma(\lambda)) - \eta(X_{0})}{\lambda} =
2 [X_1^+(0)]^2\neq 0,
$$
proving the surjectivity of the application $d\eta(X)$. 	

\subsubsection{Singular point-regular case}

Consider $X_0=(X^+_0,X^-_0)\in \Xi_1(4)$. Then $X_0^+(0)=0$, $X^-_0\neq 0$ and $0$ is a hyperbolic singular point satisfying the conditions on $\Xi_1(4)$. From Thom's transversality Theorem  and the hyperbolicity of $0$ (Proposition 3.1 of \cite{PalisMelo1982})), there exist neighborhoods $\V_{X_0}=V_{X^+_0}\times V_{X^-_0}$ of $X_0\in\Omega^r=\mathfrak{X}^r\times \mathfrak{X}^r$, $U_p$ of $0$ in $(\mathbb{R}^3,0)$ and a unique $C^r$ function $\phi: \V_{X^+_0} \to U_p \subset \R^3$  with $\phi(X^+_0)=0$, such that each $X=(X^+,X^-)\in \V_{X_0}$ satisfies
\begin{itemize}
	\item $X^+\in V_{X^+_0}$ has a unique hyperbolic singular point $\phi(X^+)=p_{X^+}\in U_p$;
	\item $X^-\in V_{X^-_0}$ is transversal to $M$ on $U_p\cap M$.
\end{itemize}
Moreover, shrinking $\mathcal{V}_{X_0}$ if necessary and using the fact that the sliding vector field $X^s$ is a linear combination of $X^+$ and $X^-$, there exists a function $\phi^s: \V_{X_0} \to U_p \cap M$  with $\phi(X_0)=0$, such that each $X\in \V_{X_0}$ satisfies
\begin{itemize}
	\item $X^s$ has a unique hyperbolic singular point $\phi^s(X)=p_{X^s}\in U_p\cap M$.
\end{itemize}

Making $\mathcal{V}_{X_0}$  even smaller if necessary, we obtain one of the following conditions:
\begin{enumerate}
	\item[(1)] if $p_{X^+}$ and $p_{X^s}$ are real singular points of $X^+$ and $X^s$, respectively, then:
	\begin{enumerate}
		\item[(1a)] the eigenvalues of $DX^+(p_{X^+})$ and $DX^s(p_{X^s})$ are of algebraic multiplicity 1;
		\item[(1b)] the eigenvectors associated to $DX^+(p_{X^+})$ and $DX^s(p_{X^s})$ are transversal to $M$ and $\partial X^s:=\{(x_1,x_2)\in M;x_2=0\}$ (the boundary of $X^s$), respectively.
		\item[(1c)] the real parts of the non-conjugated eigenvalues of $DX^+(p_{X^+})$ are distinct. 
	\end{enumerate}
	\item[(2)] if $p_{X^+}$ is a real singular point of $X^+$ and $p_{X^s}$ is a virtual singular point of $X^s$ then the generic conditions presented in $(1)$ for $X^+$ hold and $X^s$ is transversal to $\partial X^s$ on $\partial X^s\cap U_p$;
	\item[(3)] if $p_{X^+}$ is a virtual singular point of $X^+$ and $p_{X^s}$ is a real singular point of $X^s$ then the generic conditions presented in $(1)$ for $X^s$ hold and $X^+$ is transversal to $M$ on $U_p\cap M$;
	\item[(4)] if $p_{X^+}$ and $p_{X^s}$ are virtual singular points of $X^+$ and $X^s$, respectively, then the generic transversality conditions on $X^+$ and $X^s$ presented in $(2)$ and $(3)$ hold.
\end{enumerate}
We notice that in each of the bullets $(1)-(4)$ above generic conditions hold on the neighborhood $\mathcal{V}_{X_0}$. Now, following \cite{SotomayorTeixeira1988}, let $N$ be an one-dimensional section transversal to $M$ at $0$ and $\rho:U_p\longrightarrow N$ the projection on $N$ which is parallel to $M$. That is possible since $M$ is locally the hyperplane $\{(x_1,x_2,x_3)\in\mathbb{R}^3;\,x_3=0\}$. Without loss of generality, we set $N=\left<(0,0,1)\right>$ the straight line spanned by $(0,0,1)$ which is orthogonal to $M$. Therefore $\rho$ is the canonical projection on the first variable $\rho(x_1,x_2,x_3)=x_3$. Consider the function 
\[
\begin{array}{rl}
	\eta:  &\V_{X_0} \to \R	\\
	&(X^+,X^-) \mapsto \rho(p_{X^+}).
\end{array}
\]
By the construction of the function $\eta$, it is clear that $\Xi_1(4)\subset[\eta^{-1}(0) \cap \V_{X_0}]$. The inclusion $[\eta^{-1}(0) \cap \V_{X_0}]\subset \Xi_1(4)$ is also clear because every $X\in\V_{X_0}$ fulfill the generic conditions $(a)-(d)$ of $\Xi_1(4)$ and since $\eta(X)=0$, we have that $X\in \Xi_1(4)$, concluding the proof of statement $(a)$.

Now we shall prove that for all $X=(X^+, X^-)\in  \eta^{-1}(0)\cap \V_{X_0}$ we obtain that $d\eta({X}) $ is surjective. In fact, consider the smooth curve given by
\begin{equation*}\label{familia-singular-regular}
	\begin{array}{ll}
		\gamma(\lambda)  &=X_\lambda(x_1,x_2,x_3) = (X^+_\lambda(x_1,x_2,x_3), X^-_0(x_1,x_2,x_3))
	\end{array}
\end{equation*}
where $X^+_\lambda(x_1,x_2,x_3)=X^+_0(x_1,x_2,x_3-\lambda)$. Besides, we get that
\[
d\eta(X)\cdot\gamma'(0) = \dfrac{d }{d \lambda}\eta(\gamma(\lambda))\big|_{\lambda=0} = \lim_{\lambda \to 0}\dfrac{\eta(\gamma(\lambda)) - \eta(X)}{\lambda}=\lim_{\lambda \to 0}\dfrac{\eta(\gamma(\lambda))}{\lambda}=1, 
\]
since $ \eta(X)=\rho(p_{X^+})=\rho(0)=0$ and $\eta(\gamma(\lambda))=\eta(X_\lambda)=\rho(p_{X^+_\lambda})=\rho(0,0,\lambda)=\lambda$, concluding the proof of this case.

\subsubsection{Regular-regular-saddle-node singular point case}\label{secao-sela-no}
In order to prove the results for a vector field $X\in\Xi_1(5)$, we remark that $p$ is a quasi-generic singular point of $X^s$ which is a planar vector field defined on the $2-$surface $M$. Thus, we can apply the results from Sotomayor (see \cite{Sotomayor1974}) in order to exhibit the function $\eta$ fulfilling the hypotheses $(a)$ and $(b)$ for the vector field $X^s$. For the sake of completeness, next we present a summary of the construction of $\eta$, one can see \cite{Sotomayor1974} for the specific details.

Consider $X_0=(X^+_0,X^-_0)\in \Xi_1(5)$, then $X_0^+(0)\neq0$, $X^-_0(0)\neq 0$ and $0$ is a quasi-generic singular point of $X_0^s$. Due to the transversality Theorem, there exist neighborhoods $\V_{X_0}=V_{X^+_0}\times V_{X^-_0}$ of $X_0\in\Omega^r=\mathfrak{X}^r\times \mathfrak{X}^r$, $U_p$ of $0$ in $\mathbb{R}^3$ and a unique $C^r$ function $\phi: \V_{X_0} \to U_p \subset \R^3$  with $\phi(X_0)=0$, such that each $X=(X^+,X^-)\in \V_{X_0}$ satisfies $X^{\pm}(q)\neq0$ for every $q\in U_p\cap M$.

Assume that $0$ is a saddle-node of $X_0^s$ and let $\lambda_1=0$ and $\lambda_2\neq0$ be the two eigenvalues of $DX^s_0$ defined on the tangent space $T_0\mathbb{R}^3$ which is the linearization of $X_0^s$ at $0$. Let $T_1$ and $T_2$ be the eigenspaces associated to $\lambda_1$ and $\lambda_2$, respectively and set $\pi_1:T_p\longrightarrow T_1$ the projection map on $T_1$. We have $DX^s_0(v)=[V,X_0^s](p)$ where $[V,X]$ is the Lie bracket of $V$ and $X$, and for $v\in T_1$, $v\neq0$, we define $\Delta_1(X_0^s,p,v)=\pi_1(V,[V,X_0^s])(p)$. Because $p=0$ is a saddle node for $X_0^s$, there exists $v\neq0$ such that $\Delta_1(X_0^s,0,v)\neq0$. Under these conditions, let $(z_1,z_2)$ be a coordinate system around $0\in M$ satisfying $z_1(0)=z_2(0)=0$ and $\frac{\partial}{\partial z_i}(p)\in T_i$. In these coordinates, setting $X_0^s=(X^s_1,X^s_2)$ we obtain
\begin{equation}\label{sela-no-normal-form}
	\begin{array}{l}
		X^s_1(z_1,z_2)=\Delta_1 z_1^2+bz_1z_2+cz_2^2+\mathit{o}(z_1^2+z_2^2),\\
		X^s_2(z_1,z_2)=\sigma z_2+\alpha z_1^2+\beta z_1z_2+\gamma z_2^2+\mathit{o}(z_1^2+z_2^2),
	\end{array}
\end{equation}
where $\sigma(X,p)=\lambda_2$ and $\Delta_1\left(X,p,\frac{\partial}{\partial z_1}(p)\right)=\frac{\partial^2X^s_1}{\partial z_1^2}(0,0)$ is nonzero because $0$ is a saddle-node. We assume that $\sigma(X,p)<0$ and $\Delta_1\left(X,p,\frac{\partial}{\partial z_1}(p)\right)>0$.

Let $N_0^0=N_0\cap M$ and $B_0$ be neighborhoods of $0$ and $X$, respectively, where $N$ is a neighborhood of $0$ in $\mathbb{R}^3$. By continuity, there exists a neighborhood $B_0^s$ of $X_0^s$ such that, for every $Y\in B$, we get $Y^s=(Y_1^s,Y_2^s)$ is in $B_0^s$. Moreover, one obtain $\frac{\partial Y^s_2}{\partial z_2}<0$, $\Delta_1(Y^s,v_{Y^s})>0$ for a suitable $v_{Y^s}$ and $trace(DY^s(q))<0$ in $N_0^s$, where $q\in N_0^s$ is any non-generic singular point of $Y^s\in B_0^s$. Also, $\Delta_1(Y^s,q,v_{Y^s})>0$, that is, $q$ is a saddle node of $Y^s$.

Define $F:B_0^s\times N_0^s\longrightarrow \mathbb{R}$ by $F(Y^s;z_1,z_2)=Y^s_2(z_1,z_2)$. We also define functions $\mathcal{X}_1^+$, $\mathcal{X}_2^+$, and $\mathcal{X}_3^+$ as below, which allows us to define the desired function for the saddle-node case:
\begin{itemize}
	\item[(a)] Using the Implicit Function Theorem for $F$, we assure the existence of a $C^r$ function $\mathcal{X}_1^+:B_1\times I_1\longrightarrow I_2$, $B_1$ and $I_1,I_2$ neighborhoods of $X_0^s$ and $0$, respectively, such that $\mathcal{X}_1^+(X_0^s,0)=0$ and $F(Y^s;z_1,\mathcal{X}_1^+(Y^s,z_1))=Y^s_2(z_1,z_2)=0$ for $(Y^s,z_1)\in B_1\times I_1$ and $z_2\in I_2$ only if $z_2=\mathcal{X}_1^+(Y^s,z_1)$.
	
	\item[(b)] Define the $C^{r}$ function $\mathcal{X}_2^+:B_1\times I_1\rightarrow \mathbb{R}$ by $\mathcal{X}_2^+(Y^s,z_1)=Y^s_1(z_1,\mathcal{X}_1^+(z_1,Y^s))$ which satisfies $\frac{\partial \mathcal{X}_2^+}{\partial z_1}(X_0^s,0)=0$ and $\frac{\partial^2 \mathcal{X}_2^+}{\partial z_1^2}(X_0^s,0)\neq0$.
	
	\item[(c)] Applying the Implicit Function Theorem for the $C^{r-1}$ function $\frac{\partial \mathcal{X}_2^+}{\partial z_1}$ we obtain a function $\mathcal{X}_3^+:B\longrightarrow I_1$ satisfying $\mathcal{X}_3^+(X_0^s)=0$ and $\frac{\partial \mathcal{X}_2^+}{\partial z_1}(Y^s,\mathcal{X}_3^+(Y^s))=0$ for $Y^s\subset B_1$, $z_1\in I_1$ only if $z_1=\mathcal{X}_3^+(Y^s)$, where $B\subset B_0^s$.
\end{itemize}
The desired function is obtained by setting $\eta:B\longrightarrow \mathbb{R}$ by
$$
\eta(Y)=\mathcal{X}_2^+(Y^s,\mathcal{X}_3^+(Y^s))=Y^s_1(\mathcal{X}_3^+(Y^s),\mathcal{X}_1^+(Y^s,\mathcal{X}_3^+(Y^s))).
$$

We have that $\Xi_1(5)\subset[\eta^{-1}(0) \cap \V_{X_0}]$ by construction. On the other hand, $[\eta^{-1}(0) \cap \V_{X_0}]\subset \Xi_1(5)$ holds because, from the definitions of $\mathcal{X}_1^+$, $\mathcal{X}_2^+$ and $\mathcal{X}_3^+$, $Y\in B$ is such that $Y^s$ has a singular point in $N_0^s$ if, and only if, $z_2=\mathcal{X}_1^+(Y^s,z_1)$ and $\mathcal{X}_2^+(Y^s,z_1)=0$, which is a saddle-node because $\Delta_1>0$ and $\sigma<0$.

Now we shall prove that for all $X=(X^+,X^-)\in  \eta^{-1}(0)\cap \V_{X_0}$ then $d\eta({X}) $ is subjective. Indeed, the associated vector field $X^s=(X_1^s,X_2^s)$ is a smooth vector field presenting a saddle-node singular point at the origin. We notice that, assuming $X^+=(X_1^+,X_2^+,X_3^+)$ and $X^-=(X_1^-,X_2^-,X_3^-)$, in order to produce the associate sliding vector field $X^s$, the components of $X^+$ and $X^-$ are related as follows:
\begin{equation}\label{relations}
	X_1^+=\dfrac{X_3^+ X_1^s-X_1^- X_2^s}{X_3^+ X_2^--X_1^- X_3^-},\quad X_2^+=\dfrac{X_3^- X_1^s-X_2^- X_2^s}{X_3^+ X_2^--X_1^- X_3^-},
\end{equation}
where $(X_3^+ X_2^--X_1^- X_3^-)(x_1,x_2,x_3)\neq0$. Consider the smooth curve $\gamma: (-\lambda_0, \lambda_0) \to \Omega^r$ given by
\begin{equation}\label{familia-saddle-node}
	\gamma(\lambda) =X_{\lambda}= (X^+ _{\lambda}, X^-) = \left(X^+ + \lambda \left(\dfrac{X_3^+}{X_3^+ X_2^--X_1^- X_3^-},\dfrac{X_3^-}{X_3^+ X_2^--X_1^- X_3^-},0\right), X^-\right),
\end{equation}
where $X^+_0=X^+$ and $X^+$, $X^-$ are defined by the relations \eqref{relations}. One can see that, the sliding vector field $X_{\lambda}^s$ associated with $X_{\lambda}$ is now $X_{\lambda}^s= (X_1^s+\lambda, X_2^s)$. Considering \eqref{familia-saddle-node} and derivating the map $\eta$ along the curve $\gamma(\lambda)$ we get
\[
d\eta(X)\cdot\gamma'(0) = \dfrac{d }{d \lambda}\eta(\gamma(\lambda))\big|_{\lambda=0} = \lim_{\lambda \to 0}\dfrac{\eta(\gamma(\lambda)) - \eta(X_0)}{\lambda}=1,
\]
for all $(x_1, x_2)$ in a neighborhood of the origin, because $\eta(X_0)=0$ and $\eta(\gamma(\lambda)) =  X_\lambda^s(X_3^+(X),$ $X_1^+(X,$ $X_3^+(X)))=X_1^s(0,0)+\lambda=\lambda$, concluding the proof.

\subsubsection{Regular-regular-Hopf singular point case}\label{Secao-Hopf}
Consider $X_0=(X_0^+,X_0^-)\in \Xi_1(6)$ and take suitable neighborhoods and functions as in the first paragraph of the proof of the previous subsection. Let $\sigma(X_0^s,p)$ and $\Delta(X_0^s,p)$ denote the trace and the determinant of $D[X_0^s]_p$, respectively, where $X_0^s$ is the sliding vector field associated with $X_0$. Then the imaginary part of the eigenvalues of $D[X_0^s]_p$ satisfies
\[
(\sigma(X_0^s,p))^2-4\Delta(X_0^s,p)<0
\quad \text{and} \quad
\sigma(X_0^s,p)=0.
\]
Moreover, for any sliding vector field $Y^s$, by continuity, if $\sigma(Y^s,q)=0$ and $Y^s(q)=0$, then $\rho'_{X_0^s}(0)=1$ and $(\rho_{X_0^s})^{(3)}(0)\neq0$, where $\rho$ denotes the first return map associated with $p$.

Let $(z_1,z_2)$ be a coordinate system around $p$ with $z_1(p)=z_2(p)=0$. One can show from the Implicit Function Theorem that, for every neighborhood $B$ of $X$, there exists a function $P$ defined in a neighborhood $B^s$ of $X_0^s$ such that, every $Y\in B$ has an associated $Y^s=(Y_1^s,Y_2^s)\in B^s$ with equilibrium point $q=P(Y^s)$. Define the function $\eta:B\rightarrow\mathbb{R}$ as
$$
\eta(Y)=\sigma(Y^s,P(Y^s))=\dfrac{\partial Y^s_1}{\partial z_1}(P(Y^s))+\dfrac{\partial Y^s_2}{\partial z_2}(P(Y^s)),
$$
which is the trace of the matrix $DY^s_{P(Y^s)}$. Take $N^s$ a neighborhood of $0$ such that the boundary $\partial N^s$ of $N^s$ is transversal to every $Y^s\in B^s$. Under the conditions $(a)-(d)$, it can be shown that $Y^s$ has one singular point $P(Y^s)$ on $N^s$ which is generic if, and only if, $\eta(Y^s)\neq0$, and a generic, stable periodic trajectory when $\eta(Y^s)>0$. Nevertheless, we denote $M_1$ the manifold with boundary $M\setminus\partial N^s$, so the restriction $X^s_1=X_0^s|_{M_1}$ is structurally stable. Moreover, shrinking $B^s$ if necessary, every vector field $Y^s\in B^s$ is such that $Y_1=Y^s|_{M_1}$ is topologically equivalent to $X_1^s$ and it is structurally stable due to the openness of the structurally stable vector fields on $M_1$. Therefore, we have that $\Xi_1(6)=[\eta^{-1}(0) \cap \V_{X_0}]$.

It remains to prove that $d\eta_X(X)\neq0$. The associated vector field $X^s=(X_1^s,X_2^s)$ is a smooth vector field presenting a Hopf singular point at the origin. Assuming $X^+=(X_1^+,X_2^+,X_3^+)$ and $X^-=(X_1^-,X_2^-,X_3^-)$, in order to produce the associate sliding vector field $X^s$, the components of $X^+$ and $X^-$ are related as in equation \eqref{relations}. In this case, consider 
\begin{equation}\label{familia-Hopf}
	\gamma(\lambda) =X_{\lambda}= (X^+ _{\lambda}, X^-) = \left(X^+ + \lambda \left(\dfrac{z_1X_3^+}{X_3^+ X_2^--X_1^- X_3^-},\dfrac{z_1X_3^-}{X_3^+ X_2^--X_1^- X_3^-},0\right), X^-\right),
\end{equation}
where $X^+_0=X^+$. The sliding vector field $X_{\lambda}^s$ associated with $X_{\lambda}$ is now $X_{\lambda}^s= (X_1^s+\lambda z_1, X_2^s)$. Considering \eqref{familia-Hopf} we get
\[
d\eta(X)\cdot\gamma'(0) = \dfrac{d }{d \lambda}\eta(\gamma(\lambda))\big|_{\lambda=0} = \lim_{\lambda \to 0}\dfrac{\eta(\gamma(\lambda)) - \eta(X_0)}{\lambda}=1,
\]
for all $(z_1, z_2)$ in a neighborhood of the origin. In fact, $\eta(X_0)=0$ and $ \eta(\gamma(\lambda)) = \eta(X_{\lambda})= \sigma(X^s_{\lambda},P(X^s_{\lambda}))=\dfrac{\partial X_{\lambda}^+}{\partial z_1}(P(X^s_{\lambda}))+\dfrac{\partial X^-_0}{\partial z_2}(P(X^s_{\lambda}))=\lambda+\dfrac{\partial X_1^s}{\partial z_1}(0,0)+\dfrac{\partial X_2^s}{\partial z_2}(0,0)=\lambda$. This finishes the proof for this case.

\subsection{Openness of $\Xi_1\subset \Omega_1^r$} We recall that $\Xi_1 = \dis\bigcup_{i=1}^6 \Xi_1(i)$ and the openness of each subset $\Xi_1(i), i=1, \dots, 6$ in $\Omega_1^r$ follows by Section \ref{secao-submersao}.

\subsection{Density of $\Xi_1\subset \Omega_1^r$}

Let $X=(X^+,X^-)\in \Omega_1^r$ and assume that the origin is a degenerated singular point of $X$, so $X^-(0)\neq0$ and $X^+(0)$ satisfies one of the following conditions.

\begin{itemize}
	\item[(I)] If $X^+(0)\neq0$ and $X^+h(0)=0$, then we have the cases involving tangencies which are split into the following cases:
	\begin{itemize}
		\item[(a)] $(X^+)^2h(0)=(X^+)^3h(0)=(X^+)^4h(0)=0$  and $(X^+)^kh(0)\neq 0$ for $k\geq 5$ or
		\item[(b)] $(X^+)^2h(0)=0, (X^+)^3h(0)\neq 0, det(dh(0), d(X^+)h(0), d(X^+)^2h(0))=0$ and $rank(\mathrm{Hess}(X^+h(0)))=R<2$.
	\end{itemize}
	
\item[(II)] If $X^+(0)\neq 0$ and $X^+h(0)\neq 0$, then we have the cases involving two transversal vector fields at the origin generating a sliding vector field. The following cases occur:
\begin{itemize}
	\item[(c)] $0$ is a singular point of the planar vector field $X^s$ with real eigenvalues and codimension greater than or equal to two;
	\item[(d)] $0$ is a singular point of the planar vector field $X^s$ that is a weak focus of multiplicity greater than or equal to two.
\end{itemize}

	\item[(III)] If $X^+(0)=0$, then we have the case involving a singular point on $M$ as follows:
	\begin{itemize}
		\item[(e)] $0$ is a singular point of $X^+$ satisfying at least one of the following conditions: either $0$ is not a hyperbolic singular point of  $X^+$ or $X^s$, either the eigenvalues of $DX^+(0)$ or $DX^s(0)$ are of algebraic multiplicity equal or greater than 2, either the eigenvectors associated with $DX^+(0)$ are not transversal to $M$ at $0$.
	\end{itemize}
\end{itemize}
At this moment we divide the proof into two cases. The proof for tangencies-regular PSVFs, composed by the subsets $\Xi_1(1)\cup \Xi_1(2) \cup \Xi_1(3)$ and the singular point-regular/regular-regular-singular PSVFs, composed by the subsets $\Xi_1(4)\cup \Xi_1(5)\cup \Xi_1(6)$.

In the following, we consider the case (I) by the construction of a sequence $\{X_n\} \subset \Xi_1(1)\cup \Xi_1(2) \cup \Xi_1(3)$ such that $X_n \to X$. Consider the function $t\mapsto X_t(p)=\varphi_X(t,p)$ is a trajectory of $X$ starting at $p$.

Case $(a)$: In this case we have that $\varphi_X(t,p)$ has a root of multiplicity $k$ of $h\circ \varphi_X (t,p)=f(t,p)$. As a consequence of the Malgrange Preparation Theorem, see \cite{Izumiya2015, Malgrange1964}, there exists $C^r$ real-valued functions $a_1, \dots, a_k, b$ defined in a neighborhood of $(0,0) \in \R\times \R^3$ such that 
\[
b(f(t,p), p) = t^k+ \dis\sum_{i=1}^{k-1} t^ia_i(f(t,p),p),
\]
in the neighborhood of $(0,0)$. Consider a sequence $\{X_n\} \subset  \Xi_1(3)$ given by $X_n=X+\frac{1}{n}\left(0,0,x_1^3\right)$. In this way, we get $f_n(t,p) = h(X_{n,t}(p))=h(\varphi_{X_n}(t,p))$ satisfying
\[
\widetilde{b}(f_n(t,p), p) = \left(\dfrac{c_4}{n}+\mathcal{O}(p^1)\right)t^4+ \dis\sum_{i=1, i\neq 4}^{k-1} t^i\alpha_i(f_n(t,p),p)+ t^k,
\]
with $c_4 \neq 0, ||\alpha_i-a_i||= \frac{1}{n} \mathcal{O}(p^1)$ for $i\in \{1,2,\dots, k\}$. Note that by construction we have that $X_n \to X$.

\vs

Case $(b)$: Considering a local coordinate system, in this case we have that 
\begin{equation*}
\begin{array}{lllll}
X_1^+(0) \neq 0,  & X_2^+(0)=0, & X_3^+(0)=0, & \dfrac{\partial X_3^+}{\partial x_{1}}(0)=0,  &\dfrac{\partial X_3^+}{\partial x_{2}}(0)=0,  \\ \\
 \dfrac{\partial^{2} X_3^+}{\partial x_{1}^2 }(0)\neq 0, &\dfrac{\partial^{2} X_3^+}{\partial x_{2}\partial x_{3}}(0) = 0,	&\dfrac{\partial^{2} X_3^+}{\partial x_{1}\partial x_{2}}(0) = 0,  &  \dfrac{\partial^{2} X_3^+}{\partial x_{2}^2 }(0) = 0. &
	\end{array}
\end{equation*}
In this way, $rank(\mathrm{Hess}(X^+h(0)))<2$. In fact, $det(\mathrm{Hess}(Xh(0)))=\left[\dfrac{\partial^{2} X_3^+}{\partial x_{1}^2 }(0)\right] \left[\dfrac{\partial^{2} X_3^+}{\partial x_{2}^2 }(0)\right]=0$. Consider a sequence $\{X_n\}$ such that the
$X_n=X+\frac{1}{2n}\Big((0,0,x_2^2), (0,0,0)\Big)$ and note that 
$X_{n,1}^+(0) \neq 0,  X_{n,2}^+(0)=0, X_{n,3}^+(0)=0, \dfrac{\partial X_{n,3}^+}{\partial x_{1}}(0)=0, \dfrac{\partial X_{n,3}^+}{\partial x_{2}}(0)=0,  \dfrac{\partial^{2} X_{n,3}^+}{\partial x_{1}^2 }(0)\neq 0, \dfrac{\partial^{2} X_{n,3}^+}{\partial x_{2}\partial x_{3}}(0) = 0, \dfrac{\partial^{2} X_{n,3}^+}{\partial x_{1}\partial x_{2}}(0) = 0,  \dfrac{\partial^{2} X_{n,3}^+}{\partial x_{2}^2 }(0) \neq 0$. Therefore $\{X_n\} \subset  [\Xi_1(1) \cup \Xi_1(2)]$ and $X_n \to  X$.

The proof of bullets (c) and (d) of case (II) is as follows. Let $X$ be a PSVFs having as the sliding vector field $X^s$ fulfilling the conditions of the bullet (c). Then $X^s$ can be approximated by a vector field $\widetilde{X}^s$ having a quasi-generic singular point of saddle-node type (see \cite{Peixoto1962, Sotomayor1974}). Moreover, because $\widetilde{X}^s$ is $C^r-$close to $X^s$ and $X^s$ is a convex linear combination of $C^r-$functions, one can find a vector field $\widetilde{X}=(\widetilde{X}^+,X^-)$ which is $C^r-$close to $X=(X^+,X^-)$ in the product topology such that $\widetilde{X}^s$ is the sliding vector field associated to $\widetilde{X}$ (see the Subsections \ref{secao-sela-no} and \ref{Secao-Hopf}) so $\widetilde{X}\in \Xi_1(5)$. Analogously, if $X$ is a vector field having as sliding vector field $X^s$ fulfilling the conditions of bullet (d), we can obtain another vector field $\widehat{X}\in \Xi_1(6)$ close to $X$.

It remains to prove bullet (e) of case (III), so let $X^+$ be a vector field satisfying condition (e). If $0$ is a non-hyperbolic singular point for $X^+$, we can take $X_1^+$ close to $X^+$ such that $0$ is hyperbolic for $X_1^+$, see \cite{PalisMelo1982} for instance. If $X_1^s$ has $0$ as a non-hyperbolic singular point, as we done before, there exists $X_2^s$ close to $X_1^s$ having $0$ as hyperbolic singular point in such a way that $X_2^s$ is the sliding vector field associated to a vector field $X_2$ which is close to $X_1^+$ and $X^+$. If the eigenvectors associated with $DX_2^+(0)$ are not transversal to $M$ at $0$, we can use the density of transversal applications (Thom's Theorem) to assure the existence of a vector field $X_3^+$ which is sufficiently close to $X_2^+$ in such a way that is is also close to $X^+$. Finally, if the eigenvalues of $DX_3^+(0)$ or $DX_3^s(0)$ are of algebraic multiplicity equal or greater than two, we can again apply the Malgrange Preparation Theorem and the continuity of eigenvalues with respect to $DX_3^+(0)$ and $DX_3^s(0)$ to construct a sequence $X_n$ converging to $X_3$ in a way that the sequence $X_n=(X_n^+,X^-)$ converges to $X=(X^+,X^-)$ and $\{X_n\}\subset\Xi_1(4)$.

\subsection{Proof of items $(a)$ and $(c)$ of Theorem \ref{teorema-lista-cod1}}

The proof follows directly by the construction of the submersions for each case given in Subsection \ref{secao-submersao}.

\section{Proof of Theorem \ref{teorema-formas-normais}}\label{prova-teo2}
The proofs of bullets $(a)$ and $(b)$ of Theorem \ref{teorema-formas-normais} are direct from the constructions of the submersions associated with the subsets $\Xi_1(i), i=1,2,3$ given in the proof of Theorem  \ref{teorema-lista-cod1}. The proof of bullet $(c)$ is classical as done in \cite{SotomayorTeixeira1988}. We only add the generic condition $a\neq0$ which is responsible for the transversal contact between the trajectory of $X_0^-$ and $M$ at $0$.

The proof of bullets $(d)$ and $(e)$ are analogous, so we only present the proof of bullet $(d)$. They also follow from the proof of Theorem  \ref{teorema-lista-cod1} but in a less direct way than bullets $(a)$ and $(b)$, so we present it in what follows. If $X_0\in\Xi_1(5)$, then its sliding vector field $X_0^s=(X_1^s,X_2^s)$ has a saddle-node at $0$ which is assumed to be the classical saddle-node for such a smooth vector $X_0^s$. This is obtained by setting $\Delta=1$, $\sigma=-1$, and the other parameters vanishing in equation \eqref{sela-no-normal-form}. From equations \eqref{relations} and \eqref{familia-saddle-node}, we can obtain the expressions of the vector fields $X_\lambda^+$ and $X_0^-$, depending on the already fixed expressions of $X_1^s$ and $X_2^s$. Notice that, from equation \eqref{relations} with $X_3^+=0$ we must have $X_1^-X_3^-\neq0$, so we set $X_1^-=X_3^-=1$ and $X_2^-=0$. Hence, the vector field $X_0^-$ writes $X_0^-(x_1,x_2,x_3)=(1,0,1)$ and we take $X_0^+(x_1,x_2,x_3)=(X_1^+,X_2^+,X_3^+)=(-x_2,x_1^2,0)$. Finally, from equation \eqref{familia-saddle-node}, $X_\lambda^+$ writes
$$
X_\lambda^+(x_1,x_2,x_3)=X_0^+(x_1,x_2,x_3)+\lambda(0,-1,0)=(-x_2,x_1^2-\lambda,0).
$$
The expression of $X_\lambda^s(x_1,x_2)$ defined on $M$ can be obtained straightforwardly as  $X_\lambda^s(x_1,x_2)=(-x_2,x_1^2-\lambda,0)$ which coincides with equation \eqref{sela-no-normal-form} with the parameters fixed at the beginning of the proof.

\appendix 

\section{}\label{Apendice}

In this appendix, we introduce some notions on bifurcation theory that are important to the reader and that complement Section \ref{secao-preliminares}.

\subsection{A brief approach to bifurcation theory}\label{Secao-Bifurcacao}
The concept of structural stability  is based on the following definition:
\begin{definition}\label{equival}
Two vector fields $X$ and $\widetilde{X}$ in $\Omega^r$, with switching surface $M$ are said to be $M$-equivalent if there exists a homeomorphism  preserving orientation mapping trajectories of $X$ onto trajectories of $\widetilde{X}$, preserving the regions of $M$ and the sliding vector field.
\end{definition}

A vector field $X\in \Omega^r$ is structurally stable if it has a neighborhood $\V\subset \Omega^r$ such that $X$ is $M$-equivalent to every  $\widetilde{X}\in\V$.

A bifurcation phenomenon occurs when there is a qualitative change in the behavior of the family under a slight change of parameters. An important characteristic of any bifurcation is its {\it codimension}, that is the smaller number of parameters that must be varied for the bifurcation to take place.  In an arbitrary parameter space of a dimension equal to the codimension of the bifurcation, an unfolding of a bifurcation is a description of all dynamical behavior that may generically occur near the bifurcation.

\begin{remark}
We emphasize that the topological structure of $X$ is composed by the $X^+, X^-,$ \linebreak$ X^s$ and $\varphi_X$, when all these ingredients are defined, and all the iterations of these ingredients, see Subsections \ref{secao-campos-descontinuo} and \ref{secao-primeiro-retorno}. Consequently, note that the codimension of $X=(X^+,X^-)$ at $0$ is, at least, the sum of codimensions of $X^+$ and $X^-$ at $0$. To illustrate this behavior, in Example \ref{exemplo-cod-infinita} we exhibit a PSVFs that possesses codimension infinity whereas both, $X^+$ and $X^-$, are of codimension zero.
	
\begin{example}
Consider $X=(X^+,X^-)\in \Omega^r$ where $X^+(x_1,x_2,x_3)=(-1,1,-1)$ and $X^-(x_1,x_2,x_3)=(1,-1,1)$. Note that both, $X^+$ and $X^-$, are transversal to $M$, so they are of codimension zero, see \cite{SotomayorTeixeira1988}. The switching manifold $M$ coincides with the sliding region $M^s$. By \eqref{eq campo filippov}, the sliding vector fields is a two-dimensional smooth vector fields identically null, that is a codimension infinity smooth vector field. 
\label{exemplo-cod-infinita}\end{example}
\end{remark}

\noindent {\textbf{Acknowledgments.}} Rodrigo Donizete Euz\'ebio is  partially supported by Conselho Nacional de Desenvolvimento Cient\'ifico e Tecnol\'ogico (CNPq Grants 402060/2022-9 and 308652/2022-3). Durval Jos\'e Tonon is partially supported by Conselho Nacional de Desenvolvimento Cient\'ifico e Tecnol\'ogico (CNPq Grant 310362/2021-0). Marco Antonio Teixeira is partially supported by Conselho Nacional de Desenvolvimento Cient\'ifico e Tecnol\'ogico (CNPq Grant 301646/2022-8) and Funda\c c\~ao de Amparo \`a Pesquisa do Estado de S\~ao Paulo-FAPESP Grant  2018/13481-0.

\noindent {\textbf{Declarations of interest.}} The authors have no conflict of interest to declare.

\end{document}